\documentclass[11pt,oneside]{amsart}
\usepackage{amsmath,amsfonts,amsgen,amstext,amsbsy,amsopn,amsthm, amssymb, color}
\usepackage{multirow, verbatim}% cancel, mathabx}
\usepackage{bbold}
\usepackage{enumitem}
\usepackage[T1]{fontenc}
\usepackage[colorlinks=true,hyperindex=true]{hyperref}
\usepackage{tikz}
\usetikzlibrary{calc}
\usepackage{verbatim}
\usetikzlibrary{arrows,chains,matrix,positioning,scopes}
\usepackage{mathrsfs}
\usepackage{tabu} 
\makeatletter
\newcommand\red[1]{\textcolor{red}{#1}}

\newcommand\ChangeRT[1]{\noalign{\hrule height #1}}
\usepackage{booktabs,caption}
\usepackage[flushleft]{threeparttable}

\usepackage{setspace}
\usepackage{footmisc}

\usepackage{pgfplots}

\numberwithin{equation}{section}

\newcounter{smallarabics}

\newcounter{smallroman}
\newenvironment{romanenumerate}
{\begin{list}{{\normalfont\textrm{(\roman{smallroman})}}}
  {\usecounter{smallroman}\setlength{\itemindent}{0cm}
   \setlength{\leftmargin}{5ex}\setlength{\labelwidth}{4ex}
   \setlength{\topsep}{0.75\parsep}\setlength{\partopsep}{0ex}
   \setlength{\itemsep}{0ex}}}
{\end{list}}

\newcommand{\ben}{\begin{romanenumerate}}  
\newcommand{\een}{\end{romanenumerate}}

\newtheorem{theoreme}{theorem }[section]
\newtheorem{theorem}[theoreme]{Theorem}
\newtheorem{proposition}[theoreme]{Proposition}
\newtheorem{Lemma}[theoreme]{Lemma}

\newtheorem{corollary}[theoreme]{Corollary}
\newtheorem{remark}{Remark}[section]
\newtheorem{example}[theoreme]{Example}

\newcommand{\pp}{\frac{\partial \tilde{\varphi}}{\partial \overline{z}}(z)}

\newcommand{\JapA}{\Big \langle \frac{A}{L} \Big \rangle }
\newcommand{\const}{\frac{\i}{2\pi}}
\newcommand{\constL}{\frac{\i}{2\pi L}}
\newcommand{\dz}{dz\wedge d\overline{z}}
\newcommand{\I}{\mathcal{I}}

\newcommand{\CC}{\mathcal{C}}
\newcommand{\Cu}{\mathcal{C}^{1,\text{u}}(A)}

\newcommand{\Coneone}{\mathcal{C}^{1,1}(A)}
\newcommand{\Ctwo}{\mathcal{C}^{2}(A)}
\newcommand{\B}{\mathcal{B}}
\newcommand{\G}{\mathcal{G}}

\newcommand{\Pp}{P_{\rm{c}}}
\newcommand{\RH}{R}
\newcommand{\RHz}{R_0}
\newcommand{\tH}{\theta(R)}

\newcommand{\eH}{\eta(R)}
\newcommand{\chiH}{\chi(R)}
\newcommand{\eHz}{\eta(R_0)}

\usepackage{pifont}% http://ctan.org/pkg/pifont

\newcommand\nn\nonumber
\renewcommand\leq\varleq
\renewcommand\geq\vargeq

\renewcommand{\proof}{\noindent \emph{Proof. }}
 
 \newcommand{\R}{\mathbb{R}}
 \newcommand{\N}{\mathbb{N}}
\newcommand{\Z}{\mathbb{Z}} \newcommand{\C}{\mathbb{C}}
\newcommand{\D}{\mathcal{D}} \newcommand{\F}{\mathcal{E}}
 \renewcommand{\H}{\mathcal{H}}
\renewcommand{\S}{\mathcal{S}}

\newcommand{\grad}{\nabla}

\newcommand{\K}{\mathcal K}

\renewcommand{\i}{\mathrm{i}}

\renewcommand{\F} {\mathcal{F}}

\renewcommand{\epsilon}{\varepsilon}

\usepackage{pifont,xcolor}% http://ctan.org/pkg/{pifont,xcolor}
\definecolor{myblue}{RGB}{0,29,119}

\DeclareMathOperator*{\slim}{s-lim}

\usepackage{enumitem}

\newenvironment{axioms}
{\enumerate[label= $\bullet$ A\arabic*, ref=A\arabic*]}
 {\endenumerate}
\makeatletter
\newcommand\varitem[1]{\item[$\bullet$ A\arabic{enumi}\rlap{$#1$}]%
\edef\@currentlabel{A\arabic{enumi}{$#1$}}}
\makeatother

\usetikzlibrary{positioning}
\makeatletter
      \def\@setcopyright{}
      \def\serieslogo@{}
      \makeatother
\begin{document}

\author{Gol\'enia, Sylvain and Mandich, Marc-Adrien}
   \address{Sylvain Gol\'enia, Institut de Math\'ematiques de Bordeaux, 351 cours de la Lib\'{e}ration, F 33405 Talence, France}
   \email{sylvain.golenia@math.u-bordeaux.fr}
   \address{Marc-Adrien Mandich, Institut de Math\'ematiques de Bordeaux, 351 cours de la Lib\'{e}ration, F 33405 Talence, France}
   \email{marc-adrien.mandich@u-bordeaux.fr}

   % title

   \title[Propagation estimates]{Propagation estimates for one commutator regularity}

   % abstract (optional)
   \begin{abstract} 
In the abstract framework of Mourre theory, the propagation of  states is understood in terms of a conjugate operator $A$. A  powerful estimate has long been known for Hamiltonians having a good regularity with respect to $A$ thanks to the limiting absorption principle (LAP). We study the case where $H$ has less regularity with respect to $A$, specifically in a situation where the LAP and the absence of singularly continuous spectrum have not yet been established. We show that in this case the spectral measure of $H$ is a Rajchman measure and we derive some propagation estimates. One estimate is an application of minimal escape velocities, while the other estimate relies on an improved version of the RAGE formula. Based on several examples, including continuous and discrete Schr\"odinger operators, it appears that the latter propagation estimate is a new result for multi-dimensional Hamiltonians.
       \end{abstract}

   % AMS subject classifications (used in AMS journals)
%
%\renewcommand{\subjclassname}{\textup{2010} Mathematics Subject Classification} 

   % AMS keywords (used in AMS journals)

   %\date{\today}

   \maketitle 
\hypersetup{linkbordercolor=black}
\hypersetup{linkcolor=blue}
\hypersetup{citecolor=blue}
\hypersetup{urlcolor=blue}
\tableofcontents
 
\section{Introduction}

In quantum mechanics one is often interested in knowing the long-time behavior of a given state of a system. It is well-known that there exist states that to tend to remain localized in a region of space, called \textit{bound} states, while there are states that tend to drift away from all bounded regions of space, called \textit{scattering} states. The present article is concerned with the study of the latter. In particular, a propagation estimate is derived and serves to rigorously describe the long-time propagation, or behavior of these states. A classical way of obtaining a propagation estimate is by means of some resolvent estimates, or a Limiting Absorption Principle (LAP). The LAP is a powerful weighted estimate of the resolvent of an operator which implies a propagation estimate for scattering states as well as the absence of singular continuous spectrum for the system. 

The theory of Mourre was introduced by E. Mourre in \cite{m} and aims at showing a LAP. Among others, we refer to \cite{cgh,FH,GGM,HS1,jmp,S,G,GJ1} and to the book \cite{ABG} for the development of the theory. In a nutshell, Mourre theory studies the properties of a self-adjoint operator $H$, the Hamiltonian of the system, with the help of another self-adjoint operator $A$, referred to as a \textit{conjugate operator} to $H$. 
The standard Mourre theory relies on three hypotheses on the commutator of $H$ and $A$ which are, loosely speaking, that
\begin{enumerate}
\item[(M1)]  $[H, \i A]$ be positive, 
\item[(M2)]  $[H, \i A]$ be $H$-bounded, 
\item[(M3)]  $[[H,\i A], \i A]$ be $H$-bounded. 
\end{enumerate}
The main theory goes as follows:
\begin{align*}
&\underbrace{\mbox{(M1)}+\mbox{(M2)}} + \hspace{0.08cm} \mbox{(M3)} \Longrightarrow  \mbox{Resolvent estimates (LAP)} \hspace{-0.3cm}&& \Longrightarrow \mbox{Propagation estimates} \\
& \hspace{1cm} \big \Downarrow && \Longrightarrow \mbox{No singular continuous spectrum.} \\
&  \hspace{-0.18cm} \mbox{Absence of eigenvalues.} &&
\end{align*}
\noindent The purpose of the paper is to show that $\mbox{(M1)}+\mbox{(M2')} \Longrightarrow$ Weaker propagation estimates, where (M2') is slightly stronger than (M2).

We  set up notation and basic notions. For arbitrary Hilbert spaces $\F$ and $\G$, denote the bounded operators from $\F$ to $\G$ by $\B(\F,\G)$ and the compact operators from $\F$ to $\G$ by $\K(\F,\G)$. When $\F = \G$, we shall abbreviate $\B(\G) := \B(\G,\G)$ and $\K(\G) := \K(\G,\G)$. When $\G \subset \H$, denote $\mathcal{G}^*$ the antidual of $\mathcal{G}$, when we identify $\mathcal{H}$ to its antidual $\mathcal{H}^*$ by the Riesz isomorphism Theorem. Fix self-adjoint operators $H$ and $A$ on a separable complex Hilbert space $\mathcal{H}$, with domains $\mathcal{D}(H)$ and $\mathcal{D}(A)$ respectively. In Mourre theory, regularity classes are defined and serve to describe the level of regularity that $A$ has with respect to $H$. The most important of these classes are defined in Section \ref{RegularityClasses}, but we mention that they are typically distinct in applications and always satisfy the following inclusions
\begin{equation}
\label{Chain}
\CC^2(A) \subset \CC^{1,1}(A) \subset \CC^{1,\text{u}}(A) \subset \CC^{1}(A).
\end{equation}
Of these, $\CC^1(A)$ is the class with the least regularity, whereas $\CC^2(A)$ is the class with the strongest regularity. Indeed if $H \in \CC^1(A)$, then the commutator $[H,\i A]$ extends to an operator in $\mathcal{B}(\mathcal{D}(H),\mathcal{D}(H)^*)$ and is denoted $[H,\i A]_{\circ}$; whereas if $H \in \CC^2(A)$, then in addition the iterated commutator $[[H,\i A],\i A]$ extends to an operator in $\mathcal{B}(\mathcal{D}(H),\mathcal{D}(H)^*)$ and is denoted by $[[H,\i A]_{\circ},\i A]_{\circ}$ (see Section \ref{RegularityClasses}). As the $\CC^{1,\text{u}}(A)$ class plays a key role in this article we recall here its definition. We say that $H$ belongs to the $\CC^{1,\text{u}}(A)$ class if the map $t \mapsto e^{-\i t A}(H+\i)^{-1}e^{\i tA}$ is of class $\CC^1(\R; \B(\H))$, with $\B(\H)$ endowed with the norm operator topology. The standard example of operators belonging to the aforementioned classes is the following. 

\begin{example} [Continuous Schr\"odinger operators]
\label{ex:1}
Let $H_0$ be the self-adjoint realization of the Laplace operator 
$-\Delta$ in $L^2(\R^d)$. Let $Q$ be the operator of multiplication by $x=(x_1,...,x_d) \in \R^d$, and let $P:= -\i \grad$. Set
\[H:=H_0+V_{\rm sr}(Q)+V_{\rm lr}(Q),\]
where $V_{\rm sr}(x)$ and $V_{\rm   lr}(x)$ are real-valued functions that belong to $L^{\infty}(\R^d)$. Thus $V_{\rm sr}(Q)$ and $V_{\rm lr}(Q)$ are bounded self-adjoint operators on $L^2(\R^d)$ and they are respectively the short- and long-range perturbations. Suppose that $\lim V_{\rm   sr}(x) = \lim V_{\rm   lr}(x) = 0$ as $\|x\| \to +\infty$. Then $V_{\rm   sr}(Q)$ and $V_{\rm   lr}(Q)$ are $H_0$-form relatively compact operators. This notably implies that $\sigma_{\rm ess}(H)=[0,+\infty)$ by the Theorem of Weyl on relative compactness. Let $A:= (Q\cdot P + P \cdot Q)/2$ be the so-called generator of dilations. It is the standard conjugate operator to $H$. For the long-range perturbation, further assume that $x\cdot \nabla V_{\rm lr}(x)$ is a well-defined function. Table \ref{tabl1} displays Hamiltonians belonging to the classes introduced in \eqref{Chain}. The idea is clear: stronger decaying bounds on the potential imply stronger regularity. We study this example in Section \ref{Section:Examples} and prove the information reported in Table \ref{tabl1}. Finally, we should point out that many studies of (microlocal) resolvent estimates specifically for Schr\"odinger operators have been done previously; we refer to \cite{IK} and \cite{W} and references therein. 
\begin{table}[ht]
\centering
\begin{tabular}{!{\vrule width 1.0pt}c|c!{\vrule width 1.0pt}}
\ChangeRT{1.0pt}
In addition, if $\langle x\rangle V_{\rm sr}(x)$ and $x \cdot \nabla V_{\rm lr}(x)$ are & Then $H$ belongs to \\
\ChangeRT{1.0pt}
$L^{\infty}(\R^d)$ & $\CC^1(A)$ \\
$L^{\infty}(\R^d)$ and $o(1)$ & $\CC^{1,\rm{u}}(A)$ \\
$L^{\infty}(\R^d)$ and $o(\langle x\rangle^{-\epsilon})$, for some $\epsilon >0$ & $\CC^{1,1}(A)$ \\
$L^{\infty}(\R^d)$ and $O(\langle x\rangle^{-1})$  & $\CC^{2}(A)$ \\
\ChangeRT{1.0pt}
\end{tabular}
\caption{Regularity of $H$ w.r.t.\ a bound on the decay of the potential at infinity} % title of Table
\label{tabl1}
\end{table}
\end{example}

\begin{comment}
\begin{table}[ht]
\caption{Regularity of the Hamiltonian $H$ w.r.t.\ a bound on the decay of the potential at infinity} 
\begin{threeparttable}
\centering
\begin{tabular}{|c | c|} 
\hline  
& \\ [-2ex] 
In addition, if $\langle x\rangle V_{\rm sr}(x)$ and $x \cdot \nabla V_{\rm lr}(x)$ are & Then $H$ belongs to \\ [1ex] % inserts table
\hline 
 & \\ [-2ex]%% inserts single horizontal line
$L^{\infty}(\R^d)$ & $\CC^1(A)$ \\ [1ex]% inserting body of the table
$L^{\infty}(\R^d)$ and $o(1)$ & $\CC^{1,\rm{u}}(A)$ \\ [1ex]
$L^{\infty}(\R^d)$ and $o(\langle x\rangle^{-\epsilon}) ^{\dagger}$ & $\CC^{1,1}(A)$ \\ [1ex]
$L^{\infty}(\R^d)$ and $O(\langle x\rangle^{-1})$ & $\CC^{2}(A)$ \\ [1ex] % [1ex] adds vertical space
\hline %inserts single line
\end{tabular}
\begin{tablenotes}
\item $^{\dagger}$ \tiny for some $\epsilon > 0$.
\end{tablenotes}
 \end{threeparttable}
 \label{tabl1}
\end{table}
\end{comment}

%%%%%%%%%%%%%%%%%%%%%%%%%%%%%%%%%%%%%%%%%%%%%%%%%%
%%%%%%%%%%%%%%%%%%%%%%%%%%%%%%%%%%%%%%%%%%%%%%%%%%

Let $E_{\I}(H)$ be the spectral projector of $H$ on a bounded interval $\I \subset \R$. Assuming $H \in \CC^1(A)$, we say that the \textit{Mourre estimate} holds for $H$ on $\mathcal{I}$ if there is $c>0$ and $K \in \K(\H)$ such that 
\begin{equation}
\label{Mestimate}
E_{\I}(H)[H,\i A]_{\circ} E_{\I}(H) \geqslant c E_{\I}(H) + K,
\end{equation}
in the form sense on $\mathcal{H} \times \mathcal{H}$. The Mourre estimate \eqref{Mestimate} is the precise formulation of the positivity assumption (M1) alluded to at the very beginning. The Mourre estimate is localized in energy, hence it allows to infer information about the system at specific energies. Let $\mu^{A}(H)$ be the set of points where a Mourre estimate holds for $H$, i.e.\
\begin{equation*}
\mu^A(H) := \{\lambda \in \R : \exists c>0, K \in \K(\H) \ \text{and} \ \mathcal{I} \ \text{open for which} \ \eqref{Mestimate} \ \text{holds for} \ H \ \text{on} \ \mathcal{I} \ \text{and} \ \lambda \in \I \}, 
\end{equation*}

In \cite{M}, Mourre assumes roughly $H \in \CC^2(A)$ and the estimate \eqref{Mestimate} with $K=0$ to prove the following LAP on any compact sub-interval $\mathcal{J} \subset \I$:
\begin{equation}
\label{Lapsup}
\sup \limits_{x \in \mathcal{J}, \ y > 0} \| \langle A \rangle ^{-s} (H-x-\i y)^{-1} \langle A \rangle ^{-s} \| < +\infty,
\end{equation}
for all $s > 1/2$. Here $\langle A \rangle := \sqrt{1+A^2}$. We remark that if the Mourre estimate holds on $\I$ with $K=0$, then $\I$ is void of eigenvalues, as a result of the Virial Theorem \cite[Proposition 7.2.10]{ABG}. Estimate \eqref{Lapsup} can be shown to yield the following Kato-type propagation estimate:
\begin{equation}
\label{Propagation}
\sup  \limits_{\substack{\psi \in \H \\ \| \psi \| \leqslant 1 }} \int _{-\infty} ^{\infty} \|\langle A \rangle ^{-s} e^{- \i tH} E_{\mathcal{J}} (H) \psi \|^2 dt < +\infty, 
\end{equation}
which in turn implies  the absence of singular continuous spectrum on $\mathcal{J}$, e.g.\ \cite[Section XIII.7]{RS4}. The main improvement of this result is done in \cite{ABG}. The same LAP is derived assuming only $H \in \CC^{1,1}(A)$ and the estimate \eqref{Mestimate}. It is further shown that this class is optimal in the general abstract framework. Precisely in \cite[Appendix 7.B]{ABG}, there is an example of $H\in\CC^{1, \text{u}}(A)$ for which no LAP holds. However, other types of propagation estimates were subsequently derived for $H\in\CC^{1, \text{u}}(A)$, see \cite{HSS, Ri} for instance. One major motivation for wanting to obtain dynamical estimates for this class was (and still is) to have a better understanding of the nature of the continuous spectrum of $H$. The aim of this article is to provide new propagation estimates for this class of operators. We also provide a simple criterion to check if an operator belongs to the $\CC^{1, \text{u}}(A)$ class.  

Let $P_{\rm{c}}(H)$ and $P_{\rm{ac}}(H)$ respectively denote the spectral projectors onto the continuous and absolutely continuous subspaces of $H$. Our first result is the following:

\begin{theorem} 
\label{Main2} 
Let $H$ and $A$ be self-adjoint operators in a separable Hilbert space $\H$ with $H \in \CC^{1,\rm{u}}(A)$. Assume that $\I \subset \R$ is a compact interval for which $\lambda \in \mu^A(H)$ for all $\lambda \in \I$. Suppose moreover that $\ker(H-\lambda) \subset \D(A)$ for all $\lambda \in \I$. Then for all $\psi \in \H$ and all $s>0$,
\begin{equation}
\label{NewFormula3}
\lim \limits_{t \to + \infty} \| \langle A \rangle ^{-s} e^{-\i tH} P_{\rm{c}} (H) E_\I (H)  \psi \| =0.
\end{equation}
Moreover, if $W$ is $H$-relatively compact, then
\begin{equation}
\label{NewFormula4}
\lim \limits_{t \to + \infty} \| W e^{-\i tH} P_{\rm{c}} (H) E_\I (H)   \psi \| =0.
\end{equation}
In particular, if $H$ has no eigenvalues in $\I$ and $\psi \in \H$, then the spectral measure \\
$\Omega \mapsto \langle \psi, E_{\Omega \cap \I}(H) \psi \rangle$ is a Rajchman measure, i.e., its Fourier transform tends to zero at infinity.
\end{theorem}

\begin{remark}
The last part of the Theorem follows by taking $W = \langle \psi, \cdot \rangle \psi$. If $H$ has no eigenvalues in $\I$, then $P_{\rm c} (H) E_{\I}(H) = E_{\I}(H)$ and so by the Spectral Theorem, 
\[ W e^{-\i tH} P_{\rm{c}} (H) E_\I (H)   \psi  =    \psi  \times \langle \psi, e^{-\i tH} E_\I (H)   \psi \rangle  =  \psi \times \int _{\R} e^{-\i tx} d \mu _{(\psi, E_\I (H) \psi)} (x).\]
The spectral measure $\mu$ satisfies $\Omega \mapsto \mu _{(\psi, E_\I (H) \psi)}(\Omega) = \langle \psi, E_{\Omega}(H) E_\I (H) \psi \rangle = \langle \psi, E_{\Omega \cap \I}(H) \psi \rangle$. 
\end{remark}
\begin{remark}
The separability condition on the Hilbert space is used for the proof of \eqref{NewFormula4}, because the compact operator $W$ is approximated in norm by finite rank operators.
\end{remark}
\begin{remark}
Perhaps a few words about the condition $\ker(H-\lambda) \subset \D(A)$. In general, it is satisfied if $H$ has a high regularity with respect to $A$, see \cite{FMS}. Although in the present framework it is not granted, it can be valid even if $H \in \CC^1(A)$ only, as seen in \cite{JM}.
\end{remark}

This result is new to us. However, it is not strong enough to imply the absence of singular continuous spectrum for $H$. Indeed, there exist Rajchman measures whose support is a set of Hausdorff dimension zero, see \cite{B}. We refer to \cite{L} for a review of Rajchman measures. The proof of this result is an application of the minimal escape velocities obtained in \cite{Ri}. The latter is a continuation of \cite{HSS}. We refer to those articles for historical references. 

We have several comments to do concerning the various propagation estimates listed above. First, it appears in practice that $\langle A \rangle ^{-s} E_{\I}(H)$ is not always a compact operator, and so \eqref{NewFormula3} is not a particular case of \eqref{NewFormula4}. The compactness issue of $\langle A \rangle ^{-s} E_{\I}(H)$ is discussed in Section \ref{Section:Compactness}, where we study several examples including continuous and discrete Schr\"odinger operators. In all of these examples, it appears that $\langle A \rangle ^{-s} E_{\I}(H)$ is compact in dimension one, but not in higher dimensions. Second, note that \eqref{Propagation} implies \eqref{NewFormula3}. Indeed, the integrand of \eqref{Propagation} is a $L^1(\R)$ function with bounded derivative (and hence uniformly continuous on $\R$). Such functions must go to zero at infinity. On the other hand, it is an open question to know if \eqref{Propagation} is true when $H \in \CC^{1,\rm{u}}(A)$. Third, we point out that \eqref{NewFormula4} is a consequence of the Riemann-Lebesgue Lemma (see \eqref{RL} below) when $\psi = P_{\rm{ac}}(H)\psi$. This can be seen by writing the state in \eqref{NewFormula4} as $W(H+\i)^{-1} e^{-\i tH} P_{\rm c} (H) E_{\I} (H) (H+\i) \psi$ and noting that $W(H+\i)^{-1} \in \K(\H)$ and $E_{\I} (H) (H+\i) \in \B(\H)$.

Propagation estimates \eqref{NewFormula3} and \eqref{NewFormula4} cannot hold uniformly on the unit sphere of states in $\H$, for if they did, they would imply that the norm of a time-constant operator goes to zero as $t$ goes to infinity. Moving forward, we seek a propagation estimate uniform on the unit sphere and go deeper into the hypotheses. Let $\H$ be a Hilbert space. Let $H_0$ be a self-adjoint operator on $\H$, with domain $\D(H_0)$. We use standard notation and set $\H^2 := \D(H_0)$ and $\H^1 := \D(\langle H_0 \rangle ^{1/2})$, the form domain of $H_0$. Also, $\H^{-2} := (\H^2)^*$, an)d $\H^{-1} := (\H^1)^*$. The following continuous and dense embeddings hold:
\begin{equation}
\label{continuous embedding}
\H^2 \subset \H^{1} \subset \H = \H^* \subset \H^{-1} \subset \H^{-2}.
\end{equation}
%%%%%%%%%
%%%%%%%%%
These are Hilbert spaces with the appropriate graph norms. We split the assumptions into two categories: the spectral and the regularity assumptions. We start with the former.
%%%%%%%%%%%%%%%%%%%%%%%
%%%%%%%%%%%%%%%%%%%%%%%
%%%%%%%%%%%%%%%%%%%%%%%

%%%%%%%%%%%%%%%%%%%%%%%
%%%%%%%%%%%%%%%%%%%%%%%
%%%%%%%%%%%%%%%%%%%%%%%
\noindent \textbf{Spectral Assumptions:}
\begin{axioms}
\item \label{item:A1} : $H_0$ is a semi-bounded operator with form domain $\H^1$.
\item \label{item:A4} : $V$ defines a symmetric quadratic form on $\H^{1}$.
\item \label{item:A5} : $V \in \mathcal{K}(\H^{1}, \H^{-1})$.
\end{axioms}
Importantly, these assumptions allow us to define the perturbed Hamiltonian $H$. Indeed, \ref{item:A1} - \ref{item:A5} imply, by the KLMN Theorem (\cite[Theorem X.17]{RS2}), that $H := H_0 + V$ in the form sense is a semi-bounded self-adjoint operator with domain $\D(\langle H \rangle ^{1/2}) = \H^{1}$. Furthermore, we have by Weyl's Theorem that $\sigma_{\text{ess}}(H) = \sigma_{\text{ess}}(H_0)$.

Before proceeding with the other assumptions, let us take a moment to recall two well-known propagation estimates that typically hold under these few assumptions.  The first estimate is the RAGE Theorem due to Ruelle \cite{Ru}, Amrein and Georgescu \cite{AG} and Enss \cite{E}. It states that for any self-adjoint operator $H$ and any $W \in \B(\H)$ that is $H$-relatively compact, and any $\psi \in \H$,
\begin{equation}
\label{RAGE}
\lim \limits_{T \to \pm \infty} \frac{1}{T} \int_0 ^T \|W  P_{\rm{c}}(H) e^{-\i tH} \psi\| ^2 dt =0.
\end{equation}
We refer to the appendix \ref{RAGEappendix} for an observation on this Theorem. Let us go back to Example \ref{ex:1}, the case of the Schr\"odinger operators. Assuming only that the short- and long-range potentials be bounded and go to zero at infinity, we see that \ref{item:A1} - \ref{item:A5} hold. Thus $H := H_0 + V_{\rm sr}(Q) + V_{\rm   lr}(Q)$ is self-adjoint. Moreover $\mathbf{1}_{\Sigma}(Q)$ is a bounded operator that is $H$-relatively compact whenever $\Sigma \subset \R^d$ is a compact set. Hence, in this example, the above spectral assumptions and the RAGE Theorem combine to yield the following very meaningful propagation estimate:
\begin{equation}
\label{escapeCompact}
\lim \limits_{T \to \pm \infty} \frac{1}{T} \int_0 ^T \|\mathbf{1}_{\Sigma}(Q)  P_{\rm{c}}(H) e^{-\i tH} \psi\| ^2 dt =0.
\end{equation}
In words, the scattering state $P_{\rm{c}}(H) \psi$ escapes all compact sets averagely in time. The second standard estimate we wish to recall is the Riemann-Lebesgue Lemma, see e.g. \cite[Lemma 2]{RS3}. It states that for any self-adjoint operator $H$ and any $W \in \B(\H)$ that is $H$-relatively compact, and any $\psi \in \H$,
\begin{equation}
\label{RL}
\lim \limits_{t \to \pm \infty} \|W  P_{\rm{ac}}(H) e^{-\i tH} \psi\| =0.
\end{equation}
In particular, this estimate implies that the Fourier transform of the spectral measure 
\[\Omega \mapsto \langle \psi, E_{\Omega} (H) P_{\rm ac} (H) \psi \rangle = \mu_{(\psi, P_{\rm ac}(H) \psi)} (\Omega) \] 
goes to zero at infinity, i.e.\ 
\[ \int _{\R} e^{-\i tx} d\mu_{(\psi, P_{\rm ac}(H) \psi)}(x) \to 0 \quad \text{as}  \quad t \to \pm \infty.\] 
Applying the Riemann-Lebesgue Lemma to Example \ref{ex:1} gives for all compact sets $\Sigma \subset \R^d$,
\begin{equation}
\label{escapeCompact2}
\lim \limits_{t \to \pm \infty} \|\mathbf{1}_{\Sigma}(Q)  P_{\rm{ac}}(H) e^{-\i tH} \psi\| =0.
\end{equation}
Thus, the scattering state $P_{\rm{ac}}(H) \psi$ escapes all compact sets in the long run. In contrast, a basic argument such as the one given in the Appendix \ref{Heuristic} as well as estimates like \eqref{Propagation} or \eqref{NewFormula3} indicate that the scattering states tend to concentrate in regions where the conjugate operator $A$ is prevalent. We continue with the assumptions concerning the operator $H$.

\noindent \textbf{Regularity Assumptions:} There is a self-adjoint operator $A$ on $\H$ such that
\begin{axioms} 
\setcounter{enumi}{3}
\item \label{item:A3} : $e^{\i tA} \H^{1} \subset \H^{1}$ for all $t \in \R$.
\item \label{item:A2} : $H_0 \in \CC^2(A; \H^{1}, \H^{-1})$.
\item \label{item:A6} : $V \in \CC^{1,\text{u}}(A; \H^{1}, \H^{-1})$.
\varitem{'} \label{item:A6prime} : $V \in \CC^{1}(A;\H^{1}, \H^{-1})$ and $[V,\i A]_{\circ} \in \mathcal{K}(\H^{1}, \H^{-1})$.
\end{axioms}

First we note that $\CC^{\sharp}(A; \H^1, \H^{-1}) \subset \CC^{\sharp}(A)$ for $\sharp \in \{1; 1\text{,u} ; 2\}$. We refer to Section \ref{RegularityClasses} for a complete description of these classes. While \ref{item:A3} and \ref{item:A2} are standard assumptions to apply Mourre theory, \ref{item:A6} is significantly weaker. It causes $H$ to have no more than the $\CC^{1,\text{u}}(A; \H^{1}, \H^{-1})$ regularity, in which case the LAP is not always true, as mentioned previously. Proposition \ref{PropC1UU} proves the equivalence between \ref{item:A6} and \ref{item:A6prime}. In many applications, \ref{item:A6prime} is more convenient to check than \ref{item:A6}. 

Let $\mu^A(H_0)$ be the set of points where a Mourre estimate holds for $H_0$. The assumptions mentioned above imply that $\mu^A(H) = \mu ^A(H_0)$, by Lemma \ref{Lemma1}. The uniform propagation estimate derived in this paper is the following:

\begin{theorem}\label{Main}
%Suppose that $H_0 \in \CC^2(A)$ and that $V (\i - H_0)^{-1}$ is compact. Assume that $H\in \CC^{1, \text{u}}(A)$ or equivalently that $H\in \CC^1(A)$ and that \red{attention aux domaines pour l'equivalence} $[V (\i -H_0)^{-1}, \i A]_{\circ}$ is a compact operator. Suppose also that $H_0$ has a spectral gap, i.e. $\sigma(H_0) \neq \R$. Let $E \in \mu^A(H)$, and $\chi$ be a smooth cutoff function with support $\I \ni E$. If ker$(H-\lambda) \subset \mathcal{D}(A)$ for all $\lambda \in \I$, then for all $s > 1/2$ and all $\psi \in \mathcal{H}$, 
Suppose \ref{item:A1} through \ref{item:A6}. Let $\lambda \in \mu^A(H)$ be such that $\ker(H-\lambda) \subset \D(A)$. Then there exists a bounded open interval $\I$ containing $\lambda$ such that for all $s > 1/2$, 
\begin{equation}
\label{NewFormula}
\lim \limits_{T \to \pm \infty} \sup \limits_{\substack{\psi \in \H \\ \| \psi \| \leqslant 1 }} \frac{1}{T} \int_0 ^T  \| \langle A \rangle ^{-s} P_{\rm{c}}(H) E_{\I}(H)  e^{-\i t H} \psi \|^2 \ dt = 0.
\end{equation}
\end{theorem}

This formula is to be compared with \eqref{Propagation}, \eqref{NewFormula3} and \eqref{RAGE}. First note that \eqref{Propagation} implies \eqref{NewFormula}. Also, on the one hand, \eqref{NewFormula} without the supremum is a trivial consequence of \eqref{NewFormula3}. On the other hand, if \eqref{NewFormula3} held uniformly on the unit sphere, then it would imply \eqref{NewFormula}. But we saw that this is not the case. So the main gain in Theorem \ref{Main} over Theorem \ref{Main2} is the supremum. Let us further comment the supremum in \eqref{NewFormula}. This is because one can in fact take the supremum in the RAGE formula, as explained in the Appendix \ref{RAGEappendix}. The parallel with the RAGE formula (see Theorem \ref{CFKSRAGE}) raises an important concern however. The novelty of the propagation estimate \eqref{NewFormula} depends critically on the non-compactness of the operator $\langle A \rangle ^{-s} E_{\I}(H)$. As mentioned previously, it appears that $\langle A \rangle ^{-s} E_{\I}(H)$ is not always compact. Theorem \ref{Main} therefore appears to be a new result for multi-dimensional Hamiltonians. 

To summarize, the various propagation estimates discussed in the Introduction are listed in Table \ref{tab:2} according to the regularity of the potential $V$. Sufficient regularity for the free operator $H_0$ is implicit. In this table, question marks indicate open problems and R.-L.\ stands for Riemann-Lebesgue. 

\begin{table}[ht]
\centering % used for centering table
\begin{tabular}{ !{\vrule width 1.0pt} c| c| c| c| c| c| c!{\vrule width 1.0pt}} % centered columns (4 columns)
\ChangeRT{1.0pt}
$V$ is of & RAGE & R.-L. & Prop. estimates& Prop. & Kato - type  & LAP \\  % inserts table
class & formula & formula &  \eqref{NewFormula3} and  \eqref{NewFormula4}  & estimate \eqref{NewFormula} & Prop.\ estimate & \\
%heading
\ChangeRT{1.0pt}
$\CC^1(A)$ & \checkmark & \checkmark & ? & ? & ? & ? \\ 
$\CC^{1,\text{u}}(A)$ & \checkmark & \checkmark & \checkmark & \checkmark & ? & ? \\ 
$\CC^{1,1}(A)$ & \checkmark & \checkmark & \checkmark & \checkmark & \checkmark & \checkmark \\ 
$\CC^{2}(A)$ & \checkmark & \checkmark & \checkmark & \checkmark & \checkmark & \checkmark \\ 
\ChangeRT{1.0pt}
\end{tabular}
\caption{The estimates for $H$ depending on the regularity of the potential $V$} 
\label{tab:2}
\end{table}

\begin{comment}
\begin{table}[ht]
\caption{The estimates for $H$ depending on the regularity of the potential $V$} % title of Table
\centering % used for centering table
\begin{tabular}{ | c| c| c| c| c| c| c|} % centered columns (4 columns)
\hline  
 & & & & & &  \\ [-2ex]%%inserts double horizontal lines
$V$ is of & RAGE & R.-L. & Prop. estimates& Prop. & Kato - type  & LAP \\  % inserts table
class & formula & formula &  \eqref{NewFormula3} and  \eqref{NewFormula4}  & estimate \eqref{NewFormula} & Prop.\ estimate & \\ [1ex]
%heading
\hline 
 & & & & & & \\% [-2ex]% inserts single horizontal line
$\CC^1(A)$ & \checkmark & \checkmark & ? & ? & ? & ? \\ [1ex]% inserting body of the table
$\CC^{1,\text{u}}(A)$ & \checkmark & \checkmark & \checkmark & \checkmark & ? & ? \\ [1ex]
$\CC^{1,1}(A)$ & \checkmark & \checkmark & \checkmark & \checkmark & \checkmark & \checkmark \\ [1ex]
$\CC^{2}(A)$ & \checkmark & \checkmark & \checkmark & \checkmark & \checkmark & \checkmark \\ [1ex] % [1ex] adds vertical space
\hline %inserts single line
\end{tabular}
\label{tab:2}
\end{table}
\end{comment}

We underline that the LAP has been derived for several specific systems where the Hamiltonian $H$ belongs to a regularity class as low as $\CC^1(A)$, and sometimes even lower (see for example \cite{DMR}, \cite{GJ2}, \cite{JM} and \cite{Ma1} to name a few). In all these cases, a strong propagation estimate of type \eqref{Propagation} and absence of singular continuous spectrum follow. We also note that the derivation of the propagation estimate \eqref{NewFormula} is in fact very similar to the derivation of a weighted Mourre estimate which is used in the proof of a LAP for Hamiltonians with oscillating potentials belonging to the $\CC^1(A)$ class, see \cite{G} and \cite{GJ2}.

The article is organized as follows: in Section \ref{RegularityClasses}, we review the classes of regularity in Mourre theory and in particular prove the equivalence between \ref{item:A6} and \ref{item:A6prime}. In Section \ref{MourreEstimateDiscussion}, we discuss the Mourre estimate and justify that under the assumptions of Theorem \ref{Main}, $H$ and $H_0$ share the same set of points where a Mourre estimate holds. In Section \ref{Section:Examples}, we give examples of continuous and discrete Schr\"odinger operators that fit the assumptions of Theorems \ref{Main2} and \ref{Main}.  In Section \ref{PROOF2}, we prove Theorem \ref{Main2} and in Section \ref{PROOF}, we prove Theorem \ref{Main}. In Section \ref{Section:Compactness}, we discuss the compactness of the operator $\langle A \rangle ^{-s} E_{\I}(H)$. In Appendix \ref{Heuristic}, we provide a simple argument as to why we expect scattering states to evolve in the direction where the conjugate operator prevails. In Appendix \ref{RAGEappendix} we make the observation that one may in fact take a supremum in the RAGE Theorem. Finally, in Appendix \ref{Appendix} we review facts about almost analytic extensions of smooth functions that are used in the proof of the uniform propagation estimate.

\noindent \textbf{Acknowledgments:} We are very thankful to Jean-Fran\c{c}ois Bony, Vladimir Georgescu, Philippe Jaming and Thierry Jecko for precious discussions. We are very grateful to Serge Richard for explaining to us how \cite{Ri} could be used to improve our previous results. Finally, we warmly thank the anonymous referee for a meticulous reading of the manuscript and offering numerous valuable improvements. The authors were partially supported by the ANR project GeRaSic (ANR-13-BS01-0007-01).

%%%%%%%%%%%%%%%%%%%%%%%%%%%%%%%%%%%%
%%%%%%%%%%%%%%%%%%%%%%%%%%%%%%%%%%%%
%%%%%%%%%%%%%%%%%%%%%%%%%%%%%%%%%%%%

%%%%%%%%%%%%%%%%%%%%%%%%%%%%%%%%%%%%
%%%%%%%%%%%%%%%%%%%%%%%%%%%%%%%%%%%%
%%%%%%%%%%%%%%%%%%%%%%%%%%%%%%%%%%%%

%%%%%%%%%%%%%%%%%%%%%%%%%%%%%%%%%%%%
%%%%%%%%%%%%%%%%%%%%%%%%%%%%%%%%%%%%
%%%%%%%%%%%%%%%%%%%%%%%%%%%%%%%%%%%%

%%%%%%%%%%%%%%%%%%%%%%%%%%%%%%%%%%%%%%%%%%%
%%%%%%%%%%%%%%%%%%%%%%%%%%%%%%%%%%%%%%%%%%
%%%%%%%%%%%%%%%%%%%%%%%%%%%%%%%%%%%%%%%%%%%
%%%%%%%%%%%%%%%%%%%%%%%%%%%%%%%%%%%%%%%%%%%

%%%%%%%%%%%%%%%%%%%%%%%%%%%%%%%%%%%%
%%%%%%%%%%%%%%%%%%%%%%%%%%%%%%%%%%%%
%%%%%%%%%%%%%%%%%%%%%%%%%%%%%%%%%%%%

%%%%%%%%%%%%%%%%%%%%%%%%%%%%%%%%%%%%
%%%%%%%%%%%%%%%%%%%%%%%%%%%%%%%%%%%%
%%%%%%%%%%%%%%%%%%%%%%%%%%%%%%%%%%%%
%%%%%%%%%%%%%%%%%%%%%%%%%%%%%%%%%%%%
%%%%%%%%%%%%%%%%%%%%%%%%%%%%%%%%%%%%
%%%%%%%%%%%%%%%%%%%%%%%%%%%%%%%%%%%%

%%%%%%%%%%%%%%%%%%%%%%%%%%%%%%%%%%%%
%%%%%%%%%%%%%%%%%%%%%%%%%%%%%%%%%%%%
%%%%%%%%%%%%%%%%%%%%%%%%%%%%%%%%%%%%

\section{The classes of regularity in Mourre theory}
\label{RegularityClasses}

We define the classes of regularity that were introduced in \eqref{Chain}. Let $T \in \B(\H)$ and $A$ be a self-adjoint operator on the Hilbert space $\mathcal{H}$. Consider the map
\begin{align}
%\label{DefC1}
%\R \ni t & \mapsto e^{-\i tA} Te^{\i tA} \psi \in \mathcal{H}, \\
\label{DefC1u}
\R \ni t & \mapsto e^{-\i tA} Te^{\i tA}  \in \B(\H).
\end{align}
Let $k \in \N$. If the map is of class $\CC^k(\R; \B(\H))$, with $\B(\H)$ endowed with the strong operator topology, we say that $T \in \CC^k(A)$; whereas if the map is of class $\CC^k(\R ; \mathcal{B}(\mathcal{H}))$, with $\mathcal{B}(\mathcal{H})$ endowed with the operator norm topology, we say that $T \in \CC^{k,\text{u}}(A)$. Note that $\CC^{k,\text{u}}(A) \subset \CC^{k}(A)$ is immediate from the definitions. If $T \in \CC^1(A)$, then the derivative of the map \eqref{DefC1u} at $t=0$ is denoted $[T,\i A]_{\circ}$ and belongs to $\B(\H)$. Also, if $T_1,T_2 \in \B(\H)$ belong to the $\CC^1(A)$ class, then so do $T_1+T_2$ and $T_1T_2$.
We say that $T \in \Coneone$ if 
\begin{equation*}
\int_0 ^1 \Big\| [ [T, e^{\i tA}]_{\circ}, e^{ \i tA}]_{\circ} \Big \| t^{-2} dt < + \infty.
\end{equation*}
The proof that $\Ctwo \subset \Coneone \subset \Cu$ is given in \cite[Section 5]{ABG}. This yields \eqref{Chain}. 

Now let $T$  be a self-adjoint operator (possibly unbounded), with spectrum $\sigma(T)$. Let $z \in \C \setminus \sigma(T)$. We say that $T\in \CC^{\sharp}(A)$ if $(z-T)^{-1}\in \CC^{\sharp}(A)$, for $\sharp \in \{k; k\text{,u}; 1\text{,1} \}$. This definition does not depend on the choice of $z \in \C \setminus \sigma(T)$, and furthermore if $T$ is bounded and self-adjoint then the two definitions coincide, see \cite[Lemma 6.2.1]{ABG}. If $T \in \CC^1(A)$, one shows that $[T, \i A]_{\circ} \in \B(\D(T),\D(T)^*)$ and that the following formula holds:
\begin{equation}
\label{CommutatorResolvent}
[(z-T)^{-1}, \i A]_{\circ} = (z-T)^{-1} [T,\i A]_{\circ} (z-T)^{-1}.
\end{equation}

These definitions can be refined. Let $\G$ and $\H$ be Hilbert spaces verifying the following continuous and dense embeddings $\G \subset \H = \H^* \subset \G^*$, where we have identified $\H$ with its antidual $\H^*$ by the Riesz isomorphism Theorem. Let $A$ be a self-adjoint operator on $\H$, and suppose that the semi-group $\{e^{\i tA}\}_{t \in \R}$ stabilizes $\G$. Then by duality it stabilizes $\G^*$. Let $T$ be a self-adjoint operator on $\H$ belonging to $\B(\G, \G^*)$ and consider the map
\begin{equation}
\label{DefC1GG}
\R \ni t \mapsto e^{-\i tA} Te^{\i tA}  \in \B(\G,\G^*).
\end{equation}
If this map is of class $\CC^k(\R; \B(\G,\G^*))$, with $\B(\G,\G^*)$ endowed with the strong operator topology, we say that $T \in \CC^k(A; \G, \G^*)$; whereas if the map is of class $\CC^{k}(\R ; \mathcal{B}(\G,\G^*))$, with $\B(\G,\G^*)$ endowed with the norm operator topology, we say that $T \in \CC^{k,\text{u}}(A; \G, \G^*)$. If $T \in \CC^1(A; \G, \G^*)$, then the derivative of map \eqref{DefC1GG} at $t=0$ is denoted by $[T, \i A]_{\circ}$ and belongs to $\mathcal{B}(\G,\G^*)$. Moreover, by \cite[Proposition 5.1.6]{ABG}, $T \in \CC^{\sharp}(A; \G, \G^*)$ if and only if $(z-T)^{-1} \in\CC^{\sharp}(A; \G^*, \G)$ for all $z \in \C \setminus \sigma(T)$ and $\sharp \in \{ k; k\text{,u} \}$. This notably implies that $\CC^{\sharp}(A; \G, \G^*) \subset \CC^{\sharp}(A)$.

In the setting of Theorem \ref{Main}, $\G = \H^{1} := \D(\langle H_0 \rangle^{1/2})$, and $T$ stands for $H_0$, $V$ or $H$. In all cases $T\in \B(\H^1,\H^{-1})$. We also assume that $\{e^{\i tA}\}_{t \in \R}$ stabilizes $\H^1$, see \ref{item:A3}. Consider the map
\begin{equation}
\label{Map22}
\R \ni t \mapsto \langle H_0 \rangle ^{-1/2} e^{-\i tA} T e^{\i tA} \langle H_0 \rangle ^{-1/2} \in \B(\H).
\end{equation}
The latter operator belongs indeed to $\B(\H)$ since the domains concatenate as follows:
\begin{equation*}
\underbrace{\langle H_0 \rangle ^{-1/2}}_{\in \B(\H^{-1},\H)} \underbrace{e^{-\i tA}}_{\in \B(\H^{-1},\H^{-1})} \underbrace{T}_{\in \B(\H^{1}, \H^{-1})} \underbrace{e^{\i tA}}_{\in \B(\H^1,\H^1)} \underbrace{\langle H_0 \rangle ^{-1/2}}_{\in \B(\H, \H^1)}.
\end{equation*}
We remark that $T \in \CC^{k} (A; \H^1,\H^{-1})$ is equivalent to the map \eqref{Map22} being of class $\CC^k(\R; \B(\H))$, with $\B(\H)$ endowed with the strong operator topology; whereas $T \in \CC^{k,\text{u}} (A; \H^1,\H^{-1})$ is equivalent to the map being of class $\CC^k(\R; \B(\H))$, with $\B(\H)$ endowed with the norm operator topology.

In many applications, the free operator $H_0$ has a nice regularity with respect to the conjugate operator $A$, i.e.\ $H_0 \in \CC^k(A; \G, \G^*)$ for some $k \geqslant 2$ and for some $\G \subset \H$. However, the perturbation $V$ typically doesn't have very much regularity w.r.t.\ $A$ and showing that $V$ is of class $\CC^{1,\text{u}}(A; \G,\G^*)$ directly from the definition is usually not very practical. To ease the difficulty we provide the following criterion. Its proof is inspired by \cite[Lemma 8.5]{Ge}.

\begin{proposition} %[\ref{item:A6} $\Longleftrightarrow$ \ref{item:A6prime}]
\label{PropC1UU}
Suppose that $T \in \K( \H^1,\H^{-1}) \cap \CC^1(A;  \H^1,\H^{-1})$. Then $T \in \CC^{1,\rm{u}}(A; \H^1,\H^{-1})$ if and only if $[T,\i A]_{\circ} \in \K( \H^1,\H^{-1})$. 
\end{proposition}

\begin{remark}
\label{remqa}
The proof actually shows that if $T \in \mathcal{B}(\H^1,\H^{-1}) \cap \CC^1(A;  \H^1,\H^{-1})$ and $[T,\i A]_{\circ} \in \K(\H^1,\H^{-1})$, then  $T \in \CC^{1,\rm{u}}(A; \H^1,\H^{-1})$. Thus the compactness of $T$ is needed only for the reverse implication in Proposition \ref{PropC1UU}.
\end{remark}

\begin{remark}
\label{Remark2comp}
Adapting the proof of Proposition \ref{PropC1UU}, one can see that the results of Proposition \ref{PropC1UU} and Remark \ref{remqa} are still valid if $\K(\H^1,\H^{-1})$ (resp.\ $\CC^1(A;  \H^1,\H^{-1})$, resp.\ $\CC^{1,\rm{u}}(A; \H^1,\H^{-1})$, resp.\ $\mathcal{B}(\H^1,\H^{-1})$) is replaced by $\K(\H)$ (resp.\  $\CC^1(A)$, resp.\ $\CC^{1,\rm{u}}(A)$, resp.\ $\B(\H)$). 
\end{remark}

\begin{proof}
We start with the easier of the two implications, namely $T \in \CC^{1,\text{u}}(A; \H^1,\H^{-1})$ implies $[T,\i A]_{\circ} \in \K(\H^1,\H^{-1})$. Let 
\begin{align*}
\R \ni t & \mapsto \Lambda (t) := \langle H_0 \rangle ^{-1/2} e^{-\i tA} T e^{\i tA} \langle H_0 \rangle ^{-1/2} \in \B(\H).
\end{align*}
To say that $T \in \CC^{1,\text{u}}(A;\H^1,\H^{-1})$ is equivalent to $\Lambda$ being of class $\CC^1(\R,\B(\H))$, with $\B(\H)$ endowed with the norm operator topology. Since
\begin{equation*}
\langle H_0 \rangle ^{-1/2} [T,\i A]_{\circ} \langle H_0 \rangle ^{-1/2} = \lim \limits_{t \to 0} \frac{\Lambda(t)-\Lambda(0)}{t} 
\end{equation*}
holds w.r.t.\ the operator norm on $\B(\H)$ and $\Lambda(t) -\Lambda(0)$ is equal to 
\begin{equation*}
\underbrace{\langle H_0 \rangle ^{-1/2} e^{-\i tA} \langle H_0 \rangle ^{1/2}}_{\in \ \B(\H)} \underbrace{\langle H_0 \rangle ^{-1/2} T \langle H_0 \rangle ^{-1/2}}_{\in \ \K(\H)} \underbrace{\langle H_0 \rangle ^{1/2} e^{\i tA} \langle H_0 \rangle ^{-1/2}}_{\in \ \B(\H)} - \underbrace{\langle H_0 \rangle ^{-1/2} T \langle H_0 \rangle ^{-1/2}}_{\in \ \K(\H)},
\end{equation*} 
we see that $\langle H_0 \rangle ^{-1/2} [T,\i A]_{\circ} \langle H_0 \rangle ^{-1/2} \in \K(\H)$  as a norm limit of compact operators. Hence $ [T,\i A]_{\circ} \in \K(\H^1,\H^{-1})$.

We now show the reverse implication. We have to show that the map $\Lambda$ is of class $\CC^1(\R, \B(\H))$. This is the case if and only if $\Lambda$ is differentiable with continuous derivative at $t=0$. Let
\begin{equation*}
\ell (t) := \langle H_0 \rangle ^{-1/2} e^{-\i tA} [T,\i A]_{\circ} e^{\i tA} \langle H_0 \rangle ^{-1/2} \in \B(\H).
\end{equation*}
The following equality holds strongly in $\mathcal{H}$ for all $t>0$ due to the fact that $T \in \CC^1(A,\H^1,\H^{-1})$:
\begin{equation}
\label{IntegralCommutator1}
\frac{\Lambda(t) -\Lambda(0)}{t} - \ell(0) = \frac{1}{t} \int_0 ^t \langle H_0 \rangle ^{-1/2} \left(e^{-\i \tau A} [T,\i A]_{\circ} e^{\i \tau A} - [T,\i A]_{\circ}\right) \langle H_0 \rangle ^{-1/2} d\tau.
\end{equation}
Let us estimate the integrand:
\begin{align}
\begin{split}
\label{Triangle1}
& \quad \ \big\|\langle H_0 \rangle ^{-1/2}\left(e^{-\i \tau A} [T,\i A]_{\circ} e^{\i \tau A} - [T,\i A]_{\circ}\right) \langle H_0 \rangle ^{-1/2} \big\| \\
& \leqslant  \big\|\langle H_0 \rangle ^{-1/2} \left(e^{-\i \tau A} [T,\i A]_{\circ} e^{\i \tau A} - e^{-\i \tau A} [T,\i A]_{\circ}\right) \langle H_0 \rangle ^{-1/2} \big \|  \\
& \quad + \big \|\langle H_0 \rangle ^{-1/2} \left(e^{-\i \tau A} [T,\i A]_{\circ} - [T,\i A]_{\circ}\right) \langle H_0 \rangle ^{-1/2} \big\| \\
& \leqslant \Big\| \underbrace{\langle H_0 \rangle ^{-1/2} e^{-\i \tau A}\langle H_0 \rangle ^{1/2}}_{\| \cdot \| \leqslant 1} \underbrace{\langle H_0 \rangle ^{-1/2} [T,\i A]_{\circ} \langle H_0 \rangle ^{-1/2}}_{\in \ \K(\H)} \underbrace{\left(\langle H_0 \rangle ^{1/2} e^{\i \tau A}\langle H_0 \rangle ^{-1/2} -I\right)}_{\xrightarrow[]{s} 0} \Big\| \\
& \quad + \Big \|\underbrace{\left(\langle H_0 \rangle ^{-1/2}e^{-\i \tau A}\langle H_0 \rangle ^{1/2} -I \right)}_{\xrightarrow[]{s} 0} \underbrace{\langle H_0 \rangle ^{-1/2}[T,\i A]_{\circ} \langle H_0 \rangle ^{-1/2}}_{\in \ \K(\H)} \Big \|.
\end{split}
\end{align}
Thus the integrand of \eqref{IntegralCommutator1} converges in norm to zero as $t$ goes to zero. It follows that the l.h.s.\ of \eqref{IntegralCommutator1} converges in norm to zero, showing that $\Lambda'(0) = \ell (0)$. It easily follows that $\Lambda'(t) = \ell (t)$ for all $t \in \R$. Again invoking \eqref{Triangle1} shows that $\Lambda'$ is continuous at $t=0$, completing the proof.
\qed
\end{proof}

\section{A few words about the Mourre estimate}
\label{MourreEstimateDiscussion}
This section is based on the content of \cite[Section 7.2]{ABG}, where the results are presented for a self-adjoint operator $T \in \CC^1(A)$, which (we recall) contains the $\CC^1(A; \G,\G^*)$ class. 
Let $T$ be a self-adjoint operator on $\H$ with domain $\D(T) \subset \H$. Let $\G$ be a subspace such that 
\[ \D(T) \subset \G \subset \D(\langle T \rangle ^{1/2}) \subset \H = \H^* \subset \D(\langle T \rangle ^{1/2})^* \subset \G^* \subset \D(T)^*.\]
If $T \in \CC^1(A,\G, \G^*)$, then in particular $[T, \i A]_{\circ} \in \B(\G, \G^*)$. If $\I \subset \R$ is a bounded interval, then $E_{\I}(T) \in \B(\H, \G)$ and by duality $E_{\I}(T) \in \B(\G^*, \H)$. We say that the \textit{Mourre estimate} holds for $T$ w.r.t.\ $A$ on the bounded interval $\I$ if there exist $c>0$ and $K \in \K(\H)$ such that 
\begin{equation}
\label{MourreEst}
E_{\I}(T)[T, \i A]_{\circ} E_{\I}(T) \geqslant c E_{\I}(T) + K
\end{equation}
in the form sense on $\H \times \H$. Note that both the l.h.s.\  and r.h.s.\ of \eqref{MourreEst} are well-defined bounded operators on $\H$. For reminder, if this estimate holds, then the total multiplicity of eigenvalues of $T$ in $\I$ is finite by \cite[Corollary 7.2.11]{ABG}, whereas if the estimate holds with $K=0$, then $\I$ is void of eigenvalues, as a result of the Virial Theorem \cite[Proposition 7.2.10]{ABG}. We let $\mu^A(T)$ be the collection of points belonging to neighborhood for which the Mourre estimate holds, i.e.\ 
\begin{equation*}
\mu^A(T) := \{ \lambda \in \R : \exists c>0, K \in \K(\H) \ \text{and} \  \mathcal{I} \ \text{open for which} \ \eqref{MourreEst} \ \text{holds for} \ T \ \text{on} \ \mathcal{I} \ \text{and} \ \lambda \in \I \}. 
\end{equation*}
This is an open set. It is natural to introduce a function defined on $\mu^A(T)$ which gives the best constant $c>0$ that can be achieved in the Mourre estimate, i.e.\  for $\lambda \in \mu^A(T)$, let
\begin{equation*}
\varrho_T ^A (\lambda) := \sup _{\mathcal{I} \ni \lambda} \big \{ \sup \{ c \in \R : E_{\mathcal{I}}(T) [T, \i A]_{\circ} E_{\mathcal{I}}(T) \geqslant c E_{\mathcal{I}}(T) + K, \ \text{for some} \ K \in \K(\H) \} \big \}.
\end{equation*}
Equivalent definitions and various properties of the $\varrho_T ^A$ function are given in \cite[Section 7.2]{ABG}. One very useful result that we shall use is the following: 
\begin{proposition} 
\label{conjugateResolvent}
\cite[Proposition 7.2.7]{ABG}
Suppose that $T$ has a spectral gap and that $T \in \CC^1(A)$. Let $R(\varsigma) := (\varsigma -T)^{-1}$, where $\varsigma$ is a real number in the resolvent set of $T$. Then
\begin{equation}
\label{ConvertResolvent}
\varrho_{T}^A(\lambda) = (\varsigma-\lambda)^2 \varrho_{R(\varsigma)}^A((\varsigma-\lambda)^{-1}).
\end{equation}
In particular, $T$ is conjugate to $A$ at $\lambda$ if and only if $R(\varsigma) $ is conjugate to $A$ at $(\varsigma-\lambda)^{-1}$. 
\end{proposition}
As a side note, this Proposition is stated without proof in \cite{ABG}, so we indicate to the reader that it may be proven following the same lines as that of \cite[Proposition 7.2.5]{ABG} together with the following Lemma, which is the equivalent of \cite[Proposition 7.2.1]{ABG}. Denote $\I(\lambda;\epsilon)$ the open interval of radius $\epsilon$ centered at $\lambda$.
\begin{Lemma} Suppose that $T \in \CC^1(A)$. If $\lambda \notin \sigma_{\rm{ess}}(H)$, then $\varrho^A_T(\lambda) = + \infty$. If $\lambda \in \sigma_{\rm{ess}}(H)$, then $\varrho^A_T(\lambda)$ is finite and given by 
\begin{equation*}
\varrho^A_T(\lambda) = \lim \limits_{\epsilon \to 0^+} \inf \big\{ \langle \psi, [T,\i A]_{\circ} \psi \rangle : \psi \in \H, \|\psi \| =1 \ \text{and} \ E_{\I(\lambda;\epsilon)}(T) \psi = \psi \big \}.
\end{equation*}
Furthermore, there is a sequence $(\psi_n)_{n=1}^{\infty}$ of vectors such that $\psi_n \in \H$, $\|\psi_n\| \equiv 1$, $\langle \psi_n, \psi_m \rangle = \delta_{nm}$, $E_{\I(\lambda;1/n)}\psi_n = \psi_n$ and $\lim _{n\to \infty} \langle \psi_n, [T, \i A]_{\circ} \psi_n \rangle = \varrho^A_T(\lambda)$. 
\end{Lemma}

We will be employing formula \eqref{ConvertResolvent} in the proof of Theorem \ref{Main}, but for the moment we apply it to show that under the assumptions of Theorem \ref{Main}, $H$ and $H_0$ share the same points where a Mourre estimate hold. The remark is done after \cite[Theorem 7.2.9]{ABG}. Let $R(z) := (z-T)^{-1}$ and $R_0(z) := (z-T_0)^{-1}$.  
\begin{Lemma}
\label{Lemma1}
Let $T_0$, $T$ and $A$ be self-adjoint operators on $\H$. Let $T_0$ have a spectral gap, and suppose that $T,T_0 \in \CC^{1,\rm{u}}(A)$. If $R(\i)-R_0(\i) \in \K(\H)$ then $\mu^A(T)=\mu^A(T_0)$.
\end{Lemma}
\begin{remark}
\textnormal{The assumptions of Theorem \ref{Main} fulfill the requirements of this Lemma, with $(T_0,T)=(H_0,H)$. Indeed, $\D(\langle H \rangle ^{1/2}) = \D(\langle H_0 \rangle ^{1/2})$ implies the compactness of $R(\i) - R_0(\i)$: 
\begin{equation*}
R(\i) - R_0(\i) = R(\i) V R_0(\i) = \underbrace{R(\i) \langle H \rangle ^{1/2}}_{\in \B(\H)} \underbrace{\langle H \rangle ^{-1/2} \langle H_0 \rangle ^{1/2}}_{\in \B(\H)} \underbrace{\langle H_0 \rangle ^{-1/2} V \langle H_0 \rangle ^{-1/2}}_{\in \K(\H) \ \text{by} \ \ref{item:A5}} \underbrace{\langle H_0 \rangle ^{1/2} R_0(\i)}_{\in \B(\H)}.
\end{equation*}
}
\end{remark}
\begin{proof}
Firstly, the assumption that $R(\i) - R_0(\i)$ is compact implies $\sigma_{\text{ess}}(T_0) = \sigma_{\text{ess}}(T)$. Because $T_0$ has a spectral gap, $\sigma_{\text{ess}}(T_0) = \sigma_{\text{ess}}(T) \neq \R$, and therefore there exists $\varsigma \in \R \setminus (\sigma(T) \cup \sigma(T_0))$. For all $z,z' \in \R \setminus (\sigma(T) \cup \sigma(T_0))$, the following identity holds: 
\begin{equation*} 
R(z) - R_0(z) = [I +(z' -z)R(z)][R(z')-R_0(z')][I+(z'-z)R_0(z)].
\end{equation*}
Thus $R(\varsigma) - R_0(\varsigma)$ is compact.  To simplify the notation onwards, let $R_0 := R_0(\varsigma)$ and $R := R(\varsigma)$. 

Secondly, if $\lambda \in \mu^A(T_0)$, then $(\varsigma-\lambda)^{-1} \in \mu^A(R_0)$ by Proposition \ref{conjugateResolvent}, and so there is an open interval $\mathcal{I} \ni (\varsigma-\lambda)^{-1}$, $c>0$ and a compact $K$ such that
\begin{equation*}
E_{\mathcal{I}}(R_0) [R_0, \i A]_{\circ} E_{\mathcal{I}}(R_0) \geqslant c E_{\mathcal{I}}(R_0) + K.
\end{equation*}
Applying to the right and left by $\theta(R_0)$, where $\theta \in \CC_c^{\infty}(\R)$ is a bump function supported and equal to one in a neighborhood of $(\varsigma-\lambda)^{-1}$, we get
\begin{equation*}
\label{MourreEst1}
\theta(R_0) [R_0, \i A]_{\circ}\theta(R_0) \geqslant c \theta^2(R_0) + \text{compact}.
\end{equation*}
By the Helffer-Sj\"otrand formula and the fact that $R(z)-R_0(z)$ is compact for all $z \in \C \setminus \R$, we see that $\theta(R)-\theta(R_0)$ is compact, and likewise for $\theta^2(R)-\theta^2(R_0)$. Note also that $R_0 - R \in \CC^{1,\rm u}(A)$ and so by Remark \ref{Remark2comp}, $[R_0 -R, \i A]_{\circ} \in \K(\H)$. Thus exchanging $R_0$ for $R$, $\theta(R_0)$ for $\theta(R)$, and $\theta^2(R_0)$ for $\theta^2(R)$ in the previous inequality, we have
\begin{equation*}
%\label{MourreEst2}
\theta(R) [R, \i A]_{\circ}\theta(R) \geqslant c \theta^2(R) + \text{compact}.
\end{equation*}
Let $\mathcal{I}' \subset \theta^{-1}(\{1\})$. Applying $E_{\mathcal{I}'}(R)$ to the left and right of this equation shows that the Mourre estimate holds for $R$ in a neighborhood of $(\varsigma-\lambda)^{-1}$. Thus $\lambda \in \mu^A(T)$ by Proposition \ref{conjugateResolvent}, and this shows $\mu^A(T_0) \subset \mu^A(T)$. Exchanging the roles of $T$ and $T_0$ shows the reverse inclusion.
\qed
\end{proof}

\begin{comment}
We mention that under slightly stronger assumptions, an even more precise result is known:
\begin{theorem} \cite[Theorem 7.2.9]{ABG} Let $T,T_0,A$ be self-adjoint operators on $\H$ such that both $T$ and $T_0$ are of class $\CC^{1,\text{u}}(A)$. If $R(\i)-R_0(\i)$ is compact, then $\varrho_T^A = \varrho_{T_0}^A$. In particular, $\mu^A(T) = \mu^A(T_0)$. 
\end{theorem}
\end{comment}

%%%%%%%%%%%%%%%%%%%%%%%%%%%%%%%%%%%%
%%%%%%%%%%%%%%%%%%%%%%%%%%%%%%%%%%%%
%%%%%%%%%%%%%%%%%%%%%%%%%%%%%%%%%%%%
%%%%%%%%%%%%%%%%%%%%%%%%%%%%%%%%%%%%
%%%%%%%%%%%%%%%%%%%%%%%%%%%%%%%%%%%%
%%%%%%%%%%%%%%%%%%%%%%%%%%%%%%%%%%%%
\section{Examples of Schr\"odinger operators}
\label{Section:Examples}

\subsection{The case of continuous Schr\"odinger operators}
\label{ex:3}
%\normalfont
Our first application is to continuous Schr\"odinger operators. The setting has already been described in Example \ref{ex:1} for the most part. For an integer $d \geqslant 1$, we consider the Hilbert space $\H := L^2(\R^d)$.
The free operator is the Laplacian $H_0 := - \Delta = - \sum _{i=1} ^d \partial^2 / \partial x_i^2$ with domain the Sobolev space $\H^2 := \H^2(\R^d)$.
Then $H_0$ is a positive operator with purely absolutely continuous spectrum and $\sigma(H_0) = [0,+\infty)$. Let $Q$ be the operator of multiplication by $x=(x_1,...,x_d) \in \R^d$, and let $P:= -\i \grad$. Set
\[H:=H_0+V_{\rm sr}(Q)+V_{\rm lr}(Q),\]
where $V_{\rm sr}(x)$ and $V_{\rm   lr}(x)$ are real-valued functions belonging to $L^{\infty}(\R^d)$, satisfying $V_{\rm   sr}(x)$, $V_{\rm   lr}(x) = o(1)$ at infinity. Then $V_{\rm   sr}(Q)$ and $V_{\rm   lr}(Q)$ are bounded self-adjoint operators in $\H$ and $H_0$-form relatively compact operators, i.e.\ $V_{\rm   sr}(Q), V_{\rm   lr}(Q) \in \K(\H^1,\H^{-1})$, where $\H^1$ denotes the form domain of $H_0$ (i.e., the Sobolev space $\H^1(\R^d)$). The latter is a direct consequence of the following standard fact:
\begin{proposition}
\label{KnownFourierDecay}
Let $f,g$ be bounded Borel measurable functions on $\R^d$ which vanish at infinity. Then $g(Q)f(P) \in \K(L^2(\R^d))$.
\end{proposition}

Assumptions \ref{item:A1} - \ref{item:A5} are verified. We add that $\sigma_{\rm ess}(H)=[0,+\infty)$ by the Theorem of Weyl on relative compactness. 
Moving forward, we use the following results:
\begin{proposition}\cite[p.\ 258]{ABG}
\label{c1invariant}
Let $T$ and $A$ be self-adjoint operators in a Hilbert space $\mathscr{H}$ and denote $\mathscr{H}^1 := \D(\langle T \rangle^{1/2})$, the form domain of $T$, and $\mathscr{H}^{-1} := (\mathscr{H}^1)^*$. Suppose that $e^{\i tA} \mathscr{H}^1 \subset \mathscr{H}^1$. Then the following are equivalent:
\begin{enumerate}
\item $T \in \CC^1(A; \mathscr{H}^1,\mathscr{H}^{-1})$
\item The form $[T, \i A]$ defined on $\D(T) \cap \D(A)$ extends to an operator in $\B(\mathscr{H}^1, \mathscr{H}^{-1})$. 
\end{enumerate}
In this case, the derivative of the map $t \mapsto e^{-\i t A} T e^{\i t A}$ at $t=0$, which is denoted $[T, \i A]_{\circ}$ (see Section \ref{RegularityClasses}), coincides precisely with the bounded extension of the form $[T, \i A]$.
\end{proposition}

\begin{remark}
The form $[T, \i A]$ is defined for $\psi, \phi \in \D(T) \cap \D(A)$ as follows :
\[ \langle \psi, [T, \i A] \phi \rangle := \langle T^* \psi, \i A \phi \rangle - \langle A^* \psi, \i T \phi \rangle = \langle T \psi, \i A \phi \rangle - \langle A \psi, \i T \phi \rangle.\]
The last equality holds because $T$ and $A$ are assumed to be self-adjoint.
\end{remark}

\begin{proposition}\cite[p.\ 258]{ABG}
\label{c2invariant}
Under the same assumptions as that of Proposition \ref{c1invariant}, the following are equivalent:
\begin{enumerate}
\item $T \in \CC^2(A; \mathscr{H}^1,\mathscr{H}^{-1})$
\item The forms $[T,\i A]$ and $[[T, \i A]_{\circ}, \i A]$ defined on $\D(T) \cap \D(A)$ extend to operators \newline in $\B(\mathscr{H}^1, \mathscr{H}^{-1})$. 
\end{enumerate}
\end{proposition}

Let $A:= (Q\cdot P + P \cdot Q)/2$ be the generator of dilations which is essentially self-adjoint on the Schwartz space $\mathcal{S}(\R^d)$. The relation 
\begin{equation*}
(e^{\i tA} \psi) (x) = e^{td/2} \psi(e^t x), \quad \text{for all} \ \psi \in L^2(\R^d), x \in \R^d
\end{equation*}
implies that $\{e^{\i tA}\}_{t \in \R}$ stabilizes $\H^2(\R^d)$, and thus $\H^{\theta}(\R^d)$ for all $\theta \in [-2,2]$ by duality and interpolation. Thus \ref{item:A3} holds.
A straightforward computation gives
\begin{equation}
\label{commutatorH0}
\langle \psi, [H_0, \i A] \phi \rangle = \langle \psi, 2H_0 \phi \rangle 
\end{equation}
for all $\psi,\phi \in \D(H_0) \cap \D(A)$. We see that $[H_0,\i A]$ extends to an operator in $\B(\H^1,\H^{-1})$, thereby implying that $H_0 \in \CC^1(A;\H^1,\H^{-1})$ by Proposition \ref{c1invariant}. The strict Mourre estimate holds for $H_0$ with respect to $A$ on all intervals $\I$ verifying $\overline{\I} \subset (0,+\infty)$. In particular, $\mu^A(H_0) = (0, +\infty)$. 

From \eqref{commutatorH0} it is immediate that
\begin{equation*}
\langle \psi, [[H_0, \i A]_{\circ}, \i A] \phi \rangle = \langle \psi, 4 H_0 \phi \rangle
\end{equation*}
holds for all $\psi,\phi \in \D(H_0) \cap \D(A)$. In particular, $[[H_0, \i A]_{\circ}, \i A]$ extends to an operator in $\B(\H^1,\H^{-1})$. Applying Proposition \ref{c2invariant} gives $H_0 \in \CC^2(A;\H^1,\H^{-1})$ and \ref{item:A2} is fulfilled. 

We now examine the commutator between the the full Hamiltonian $H := H_0 + V_{\rm lr}(Q)+V_{\rm sr}(Q)$ and $A$. Since $V_{\rm lr}(x)$ and $V_{\rm sr}(x)$ are assumed to be real-valued bounded functions, $\D(H) = \D(H_0)$.  So we consider the form 
\begin{equation}
\label{commutatorHH}
\langle \psi, [H, \i A] \phi \rangle = \langle \psi, [H_0, \i A] \phi \rangle + \langle \psi, [V_{\rm lr}(Q), \i A] \phi \rangle + \langle \psi, [V_{\rm sr}(Q), \i A] \phi \rangle
\end{equation}
defined for all $\psi,\phi \in \D(H_0) \cap \D(A)$. By linearity, we may treat each commutator form separately and the one with $[H_0, \i A]$ has already been dealt with. 

For the long-range potential, we now additionally assume that $x \cdot \grad V_{\rm lr}(x)$ exists as a function and belongs to $L^{\infty}(\R^d)$. A computation gives 
\[ \langle \psi, [V_{\rm lr}(Q), \i A] \phi \rangle = - \langle \psi, Q \cdot \grad V_{\rm lr}(Q) \phi \rangle, \]
for all $\psi, \phi \in \D(H_0) \cap \D(A)$. This shows that $[V_{\rm lr}(Q), \i A]$ extends to an operator in $\B(\H^1, \H^{-1})$. 

For the short range potential, we now additionally assume that $\langle x \rangle V_{\rm sr}(x)$ belongs to $L^{\infty}(\R^d)$. For $\psi, \phi \in \D(H_0) \cap \D(A)$, we have
\begin{align*}
\langle \psi, [ V_{\rm sr}(Q), \i A ] \phi \rangle &= \langle V_{\rm sr}(Q) \psi, \i A \phi \rangle + \langle \i A \psi, V_{\rm sr}(Q)  \phi \rangle \\
&= \langle \langle Q \rangle V_{\rm sr}(Q)  \psi, \langle Q \rangle ^{-1} (\i Q \cdot P + d/2) \phi \rangle + \langle  \langle Q \rangle ^{-1} (\i Q \cdot P + d/2) \psi, \langle Q \rangle V_{\rm sr}(Q)  \phi \rangle.
\end{align*}
We handle the operator in the first inner product on the r.h.s.\ of the previous equation. Note that $\langle Q \rangle ^{-1} (\i Q \cdot P + d/2) \in \B(\H^1, \H)$ and $\langle Q \rangle V_{\rm sr}(Q) \in \B(\H, \H^{-1})$. Thus 
\begin{equation}
\label{firstoperator}
\langle Q \rangle V_{\rm sr}(Q) \times \langle Q \rangle ^{-1} (\i Q \cdot P + d/2)
\end{equation}
belongs to $\B(\H^1, \H^{-1})$. The operator in the second inner product on the r.h.s.\ of the previous equation also belongs to $\B(\H^1, \H^{-1})$, because it is the adjoint of \eqref{firstoperator}.  We conclude that $[V_{\rm sr}(Q), \i A]$ extends to an operator in $\B(\H^1, \H^{-1})$. 

Putting everything together, we see that the form $[H, \i A]$ given in \eqref{commutatorHH} extends to an operator in $\B(\H^1, \H^{-1})$. Applying Proposition \ref{c1invariant} gives $H \in \CC^1(A; \H^1, \H^{-1})$. Writing $V_{\rm lr}(Q) + V_{\rm sr}(Q) = H - H_0$, we see that both $V_{\rm lr}(Q)$ and $V_{\rm sr}(Q)$ belong to $\CC^1(A; \H^1, \H^{-1})$. Of course we note that the bounded extensions of $[V_{\rm lr}(Q), \i A]$ and $[V_{\rm sr}(Q), \i A]$ are precisely $[V_{\rm lr}(Q), \i A]_{\circ}$ and $[V_{\rm sr}(Q), \i A]_{\circ}$ respectively.

If $x \cdot \grad V_{\rm lr}(x) = o(1)$  at infinity is further assumed, then $\langle H_0 \rangle ^{-1/2} Q \cdot  \grad V_{\rm lr}(Q) \langle H_0 \rangle ^{-1/2}$ belongs to $\K(L^2(\R^d))$ by Proposition \ref{KnownFourierDecay}. In other words, $[V_{\rm lr}(Q), \i A]_{\circ}  =  -Q \cdot  \grad V_{\rm lr}(Q) \in \K(\H^1, \H^{-1})$. So under this additional assumption, $V_{\rm lr}(Q) \in \CC^{1,\rm{u}}(A; \H^1, \H^{-1})$ by Proposition \ref{PropC1UU}.

If $\langle x \rangle V_{\rm sr}(x) = o(1)$  at infinity is further assumed, then $\langle H_0 \rangle ^{-1/2} \langle Q \rangle V_{\rm sr}(Q)$ belongs to $\K(L^2(\R^d))$ by Proposition \ref{KnownFourierDecay}. In other words $\langle Q \rangle V_{\rm sr}(Q) \in \K(\H, \H^{-1})$. So \eqref{firstoperator} and its adjoint belong to $\K(\H^1,\H^{-1})$. Thus $[V_{\rm lr}(Q), \i A]_{\circ} \in \K(\H^1, \H^{-1})$. Applying Proposition \ref{PropC1UU} gives $V_{\rm sr}(Q) \in \CC^{1,\rm{u}}(A; \H^1, \H^{-1})$. Assumption \ref{item:A6} is fulfilled. 

The above calculations along with Theorems \ref{Main2} and \ref{Main} yield the following specific result for continuous Schr\"odinger operators:
\begin{theorem}
\label{ContSchrodingerThm}
Let $\H =  L^2(\R^d)$, $H:=H_0+V_{\rm sr}(Q)+V_{\rm lr}(Q)$ and $A$ be as above, namely
\begin{enumerate}
\item $H_0 =-\Delta$ and $A = (Q \cdot P + P \cdot Q)/2$,
\item $V_{\rm sr}(x)$ and $V_{\rm lr}(x)$ are real-valued functions in $L^{\infty}(\R^d)$,
\item $\lim V_{\rm sr}(x) = \lim V_{\rm lr}(x) =0$ as $\| x\| \to +\infty$,
\item \label{Assumption4} $\lim \langle x \rangle V_{\rm sr}(x) = 0$ as $\| x\| \to +\infty$, and
\item \label{Assumption5} $x\cdot \grad V_{\rm lr}(x)$ exists as a function, belongs to $L^{\infty}(\R^d)$, and $\lim x\cdot \grad V_{\rm lr}(x) = 0$ as $\| x\| \to +\infty$.
\end{enumerate}
Then for all $\lambda \in (0,+\infty)$ there is a bounded open interval $\I$ containing $\lambda$ such that for all $s>0$ and $\psi \in \H$, propagation estimates \eqref{NewFormula3} and \eqref{NewFormula4} hold, and for all $s>1/2$, estimate \eqref{NewFormula} holds.
\end{theorem}

\begin{remark}
As discussed above, Assumptions (1) - (5) of Theorem \ref{ContSchrodingerThm} imply that $V_{\rm sr}(Q)$ and $V_{\rm lr}(Q)$ belong to $\CC^{1,\rm{u}}(A; \H^1,\H^{-1})$. In particular $H \in \CC^{1,\rm{u}}(A)$. Moreover, $\mu^A(H) = \mu^A(H_0) = (0,+\infty)$, by Lemma \ref{Lemma1}. 
\end{remark}

\begin{remark}
Notice that the condition $\ker(H-\lambda) \subset \D(A)$ that appears in the formulation of Theorems \ref{Main2} and \ref{Main} is totally absent here. This is because under the assumptions $\lim \langle x \rangle V_{\rm sr}(x) = \lim x\cdot \grad V_{\rm lr}(x) = 0$ as $\| x\| \to +\infty$, it is well-known from research in the sixties that the continuous Schr\"odinger operator $H$ does not have any eigenvalues in $[0,+\infty)$, see articles by Kato \cite{K2}, Simon \cite{Si} and Agmon \cite{A}.
\end{remark}

To end this section, we justify the statements made in the last two rows of Table \ref{tabl1}. We refer the reader to \cite[Theorem 7.6.8]{ABG} for the proof that $H := H_0 + V_{\rm sr}(Q) + V_{\rm lr}(Q)$ belongs to $\CC^{1,1}(A)$ whenever $\langle x \rangle V_{\rm sr}(x)$ and $x \cdot \grad V_{\rm lr}(x)$ belong to $L^{\infty}(\R^d)$ and satisfy $o(\langle x \rangle ^{-\epsilon})$ at infinity for some $\epsilon > 0$. As for the statement concerning the $\CC^{2}(A)$ regularity, it remains to prove that if some arbitrary potential $V(x)$ is a bounded real-valued function with $\langle x \rangle ^2 V(x) \in L^{\infty}(\R^d)$, then $H := H_0 + V(Q)$ belongs to $\CC^{2}(A)$. Specifically, under these assumptions for $V$, we will show that $H$ belongs to $\CC^2(A; \H^1,\H^{-1}) \subset \CC^2(A)$. 

\begin{comment}
\begin{proposition}\cite[p.\ 258]{ABG}
\label{c3invariant}
Let $T$ and $A$ be self-adjoint operators in a Hilbert space $\mathscr{H}$ and denote $\mathscr{H}^2 := \D(T)$, the domain of $T$, and $\mathscr{H}^{-2} := (\mathscr{H}^2)^*$. Suppose that $e^{\i tA} \mathscr{H}^2 \subset \mathscr{H}^{-2}$. Then the following are equivalent:
\begin{enumerate}
\item $T \in \CC^2(A; \mathscr{H}^2,\mathscr{H}^{-2})$
\item The forms $[T,\i A]$ and $[[T, \i A]_{\circ}, \i A]$ defined on $\D(T) \cap \D(A)$ extend to operators \newline 
in $\B(\mathscr{H}^2, \mathscr{H}^{-2})$. 
\end{enumerate}
\end{proposition}
\end{comment}

Since $V(x)$ is bounded and real-valued, $\D(H) = \D(H_0)$ and so we consider first the form
\begin{equation*}
\label{commutatorHHH}
\langle \psi, [H, \i A] \phi \rangle = \langle \psi, [H_0, \i A] \phi \rangle + \langle \psi, [V(Q), \i A] \phi \rangle 
\end{equation*}
defined for $\psi,\phi \in \D(H_0) \cap \D(A)$.  The calculations from above imply that $[H, \i A]$ extends to an operator in $\B(\H^1, \H^{-1})$. Second we consider the form 
\begin{equation*}
\label{secondcommutatorHHH}
\langle \psi, [[H, \i A]_{\circ}, \i A] \phi \rangle = \langle \psi, [[H_0, \i A]_{\circ}, \i A] \phi \rangle + \langle \psi, [[V(Q), \i A]_{\circ}, \i A] \phi \rangle 
\end{equation*}
defined for $\psi, \phi \in \D(H_0) \cap \D(A)$. The first commutator form on the r.h.s.\ of this equation simplifies to $\langle \psi, 4 H_0 \phi \rangle$, while the for second commutator form we have :
\begin{align}
\langle \psi, [[V(Q), \i A]_{\circ}, \i A] \phi \rangle &=  \big \langle [V(Q), \i A]_{\circ} ^* \psi, (\i A) \phi \big \rangle + \big \langle (\i A) \psi, [V(Q), \i A]_{\circ} \phi \big\rangle \nonumber \\
&= \big \langle [V(Q), \i A]_{\circ}  \psi, (\i A) \phi \big \rangle + \big \langle (\i A) \psi, [V(Q), \i A]_{\circ} \phi \big\rangle.
\label{commut2}
 \end{align}
Let us first have a look at the second term $\big \langle (\i A) \psi, [V(Q), \i A]_{\circ} \phi \big\rangle$. Since $\langle x \rangle V(x) \in L^{\infty}(\R^d)$, this term is equal to 
$$\big \langle (\i A) \psi, \underbrace{\langle Q \rangle V(Q) \langle Q \rangle ^{-1} (\i Q \cdot P +d/2)}_{ := \ \Delta \ \in \ \B(\H^1,\H^{-1})} \phi \big\rangle + \big\langle (\i A) \psi, \underbrace{(-\i  P \cdot Q +d/2) \langle Q \rangle ^{-1} \langle Q \rangle V(Q)}_{ = \ \Delta^* \ \in \ \B(\H^1,\H^{-1})} \phi \big\rangle.$$
Since $\langle x \rangle ^{2} V(x) \in L^{\infty}(\R^d)$, this is also equal to 
$$\big \langle \langle Q \rangle ^{-1} (\i A) \psi, \underbrace{\langle Q \rangle ^{2} V(Q) \langle Q \rangle ^{-1} (\i Q \cdot P +d/2)}_{ := \ \tilde{\Delta} \ \in \ \B(\H^1,\H^{-1})} \phi \big\rangle +\big \langle (\i A) \psi, \underbrace{(-\i  P \cdot Q +d/2) \langle Q \rangle ^{-2} \langle Q \rangle ^2 V(Q)}_{ = \ \Delta^* \ \in \ \B(\H^1,\H^{-1})} \phi \big\rangle. $$
Let $\mathcal{A}_Q := \langle Q \rangle ^{-1} (\i A) \in \B(\H^1,\H^{-1})$. Using the fact that $[(- \i P\cdot Q +d/2), \langle Q \rangle ^{-1} ] = [-\i P,  \langle Q \rangle ^{-1} ] \cdot Q = \i |Q|^2 \langle Q \rangle ^{-3}$ , we get that $\big \langle (\i A) \psi, [V(Q), \i A]_{\circ} \phi \big\rangle$ is equal to 
$$\big \langle \mathcal{A}_Q \psi, \tilde{\Delta} \phi \big\rangle + \big \langle \mathcal{A}_Q \psi, \underbrace{(-\i  P \cdot Q +d/2) \langle Q \rangle ^{-1} \langle Q \rangle ^2 V(Q)}_{ \in \ \B(\H^1,\H^{-1})} \phi \big\rangle + \big \langle \mathcal{A}_Q \psi, \underbrace{\i  |Q|^2 \langle Q \rangle ^{-1} V(Q)}_{\in \ \B(\H)}\phi \big\rangle. $$
For the first term of \eqref{commut2}, we note that it is equal to $\overline{\big \langle  (\i A) \phi,[V(Q), \i A]_{\circ}  \psi \big \rangle}$. Performing the same calculations as we just did shows that the commutator $[[V(Q), \i A]_{\circ}, \i A]$ extends to a bounded operator in $\B(\H^1,\H^{-1})$. 
Hence $[[H, \i A]_{\circ}, \i A]$ extends to a bounded operator in $\B(\H^1,\H^{-1})$ and by Proposition \ref{c2invariant} this implies that $H \in \CC^2(A ; \H^1, \H^{-1})$.

\subsection{The case of discrete Schr\"odinger operators}
\label{ex:2}
Our second application is to discrete Schr\"odinger operators. For an integer $d \geqslant 1$, let $\H := \ell^2(\Z^d)$
The free operator is the discrete Laplacian, i.e.\ $H_0 := \Delta \in \B(\H)$, where 
\begin{equation}
\label{DeltaDef}
(\Delta \psi)(n) := \sum_{m : \|m-n\|=1} \psi(n) - \psi(m).
\end{equation}
Here we have equipped $\Z^d$ with the following norm: for $n=(n_1,...,n_d)$, $\|n\| := \sum_{i =1}^d |n_i|$. It is well-known that $\Delta$ is a bounded positive operator on $\H$ with purely absolutely continuous spectrum, and $\sigma(\Delta) = \sigma_{\rm{ac}}(\Delta) = [0,4d]$. Let $V$ be a bounded real-valued function on $\Z^d$ such that $V(n) \to 0$ as $\|n\| \to +\infty$. Then $V$ induces a bounded self-adjoint compact operator on $\H$ as follows, $(V\psi)(n) := V(n) \psi(n)$. Recall that a multiplication operator $V$ on $\ell^2(\Z^d)$ is compact if and only if $V(n) \to 0$ as $\|n\| \to +\infty$. Assumptions \ref{item:A1} - \ref{item:A5} are verified. Set $H := H_0 + V$. Then $H$ is a bounded self-adjoint operator and $\sigma_{\text{ess}}(H) = [0,4d]$. 

To write the conjugate operator, we need more notation. Let $S=(S_1,...,S_d)$, where, for $1 \leqslant i \leqslant d$, $S_i$ is the shift operator given by 
\begin{equation*}
(S_i \psi)(n) := \psi(n_1,...,n_i -1,...,n_d), \quad \text{for all} \ n \in \Z^d \ \text{and} \ \psi \in \H. 
\end{equation*} 
Let $N=(N_1,...,N_d)$, where, for $1 \leqslant i \leqslant d$, $N_i$ is the position operator given by 
\begin{equation*}
(N_i \psi)(n) := n_i \psi(n), \quad \text{with domain} \quad \D(N_i) := \bigg \{ \psi \in \H : \sum_{n \in \Z^d} |n_i \psi(n)|^2 < +\infty \bigg \}. 
\end{equation*} 
The conjugate operator, denoted by $A$, is the closure of the following operator
\begin{equation}
\label{Adef}
A_0 := \frac{\i}{2} \sum _{i=1}^d (S_i -S_i^*)N_i + N_i (S_i -S_i^*), \quad \text{with domain} \quad \D(A_0) := \ell_0(\Z^d),
\end{equation} 
the sequences with compact support. The operator $A$ is self-adjoint, see \cite{BS} and \cite{GGo}. That $\{e^{\i tA}\}_{t \in \R}$ stabilizes the form domain of $H_0$ is a triviality, because $\D(H_0) = \H$. So Assumption \ref{item:A3} is true. 

Next, we study the commutator between $H_0$ and $A$. A calculation shows that 
\begin{equation}
\label{CommutatorLaplacian}
\langle \psi, [H_0, \i A] \psi \rangle = \langle \psi, \sum_{i=1}^d \Delta_i (4- \Delta_i) \psi \rangle,
\end{equation}
for all $\psi \in \ell_0(\Z^d)$. Here $\Delta_i := 2-S_i-S_i^*$. Since $H_0$ is a bounded self-adjoint operator, \eqref{CommutatorLaplacian} implies that $H_0 \in \CC^1(A)$, thanks to a simple criterion for such operators, see \cite[Lemma 6.2.9]{ABG} and \cite[Theorem 6.2.10]{ABG}. We could also have invoked Proposition \ref{c1invariant}, but that is a bit of an overkill. An easy induction shows that $H_0 \in \CC^k(A)$ for all $k \in \N$. In particular, Assumption \ref{item:A2} holds.
By \eqref{CommutatorLaplacian} and \cite[Theorem 8.3.6]{ABG}, we have that 
\begin{equation}
\label{SetDelta}
\mu^A(H_0) = [0,4d] \setminus \{4k : k=0,...,d\}.
\end{equation}

Let us now study the commutator between $V$ and $A$. Let $\tau_i V$ be the shifted potential acting as follows:
\begin{equation*}
[(\tau_i V)\psi](n) := V(n_1,...,n_i -1,...,n_d) \psi(n), \quad \text{for all} \ \psi \in \H.
\end{equation*}
Define $\tau_i^*V$ correspondingly. A straightforward computation gives
\begin{equation*}
\langle \psi, [V, \i A ] \psi \rangle = \sum_{i=1}^d \Big \langle \psi, \Big( (2^{-1}+N_i)(V-\tau_i ^*V)S_i^* + (2^{-1}-N_i)(V-\tau_i V)S_i \Big) \psi \Big \rangle, 
\end{equation*}
for all $\psi \in \ell_0(\Z^d)$. If $\sup_{n \in \Z^d} | n_i(V-\tau_i V)(n)| <+\infty$ is assumed for all $1 \leqslant i \leqslant d$, we see that $V \in \CC^1(A)$. The bounded extension of the form $[V, \i A]$ is precisely $[V, \i A]_{\circ}$. If $\lim | n_i(V-\tau_iV)(n)| = 0$ as $\|n\| \to +\infty$ for all $1\leqslant i \leqslant d$ is further assumed, then $[V,\i A]_{\circ} \in \K(\H)$. This is equivalent to $V \in \CC^{1,\text{u}}(A)$, by Remark \ref{Remark2comp}. Thus \ref{item:A6} is fulfilled. 

The above calculations along with Theorems \ref{Main2} and \ref{Main} yield the following specific result for discrete Schr\"odinger operators:
\begin{theorem}
%\label{ContSchrodingerThm}
Let $\H = \ell^2(\Z^d)$, $H:=H_0+V$ and $A$ be as above, namely
\begin{enumerate}
\item $H_0$ is given by \eqref{DeltaDef} and $A$ is the closure of \eqref{Adef},
\item $V(n)$ is a bounded real-valued function defined on $\Z^d$,
\item \label{Assumption33} $\lim V(n) = 0$ as $\| n\| \to +\infty$, and
\item \label{Assumption44} $\lim |n_i(V-\tau_iV)(n)| = 0$ as $\|n\| \to +\infty$ for all $1\leqslant i \leqslant d$.
\end{enumerate}
Then for all $\lambda \in [0,4d] \setminus \{4k : k=0,...,d\}$ there is a bounded open interval $\I$ containing $\lambda$ such that for all $s>0$ and $\psi \in \H$, propagation estimates \eqref{NewFormula3} and \eqref{NewFormula4} hold, and for all $s>1/2$, estimate \eqref{NewFormula} holds.
\end{theorem}

\begin{remark}
As seen above, Assumptions (1) - (4) imply that $V$ belongs to $\CC^{1,\rm{u}}(A)$. In particular $H \in \CC^{1,\rm{u}}(A)$. Moreover, $\mu^A(H) = \mu^A(H_0) = [0,4d] \setminus \{4k : k=0,...,d\}$, by Lemma \ref{Lemma1} and \eqref{SetDelta}.
\end{remark}

\begin{remark}
As in the continuous operator case, the condition $\ker(H-\lambda) \subset \D(A)$ holds here for all $\lambda \in \mu^{A}(H)$. Indeed, if $\psi \in \ker(H-\lambda)$ and $\lambda \in \mu^A(H)$, then $\psi$ decays sub-exponentially, see \cite[Theorem 1.5]{Ma2}. Under Assumptions \eqref{Assumption33} and \eqref{Assumption44}, the absence of positive eigenvalues holds for one-dimensional discrete Schr\"odinger operators, by \cite[Theorem 1.3]{Ma2}. To our knowledge, the absence of positive eigenvalues under Assumptions \eqref{Assumption33} and \eqref{Assumption44} is an open problem for multi-dimensional discrete Schr\"odinger operators on $\Z^d$.
\end{remark}

%%%%%%%%%%%%%%%%%%%%%%%%%%%%%%%%%%%%
%%%%%%%%%%%%%%%%%%%%%%%%%%%%%%%%%%%%
%%%%%%%%%%%%%%%%%%%%%%%%%%%%%%%%%%%%

%%%%%%%%%%%%%%%%%%%%%%%%%%%%%%%%%%%%
%%%%%%%%%%%%%%%%%%%%%%%%%%%%%%%%%%%%
%%%%%%%%%%%%%%%%%%%%%%%%%%%%%%%%%%%%
%%%%%%%%%%%%%%%%%%%%%%%%%%%%%%%%%%%%
%%%%%%%%%%%%%%%%%%%%%%%%%%%%%%%%%%%%
%%%%%%%%%%%%%%%%%%%%%%%%%%%%%%%%%%%%

%%%%%%%%%%%%%%%%%%%%%%%%%%%%%%%%%%%%
%%%%%%%%%%%%%%%%%%%%%%%%%%%%%%%%%%%%
%%%%%%%%%%%%%%%%%%%%%%%%%%%%%%%%%%%%

%%%%%%%%%%%%%%%%%%%%%%%%%%%%%%%%%%%%
%%%%%%%%%%%%%%%%%%%%%%%%%%%%%%%%%%%%
%%%%%%%%%%%%%%%%%%%%%%%%%%%%%%%%%%%%
%%%%%%%%%%%%%%%%%%%%%%%%%%%%%%%%%%%%
%%%%%%%%%%%%%%%%%%%%%%%%%%%%%%%%%%%%
%%%%%%%%%%%%%%%%%%%%%%%%%%%%%%%%%%%%

%%%%%%%%%%%%%%%%%%%%%%%%%%%%%%%%%%%%
%%%%%%%%%%%%%%%%%%%%%%%%%%%%%%%%%%%%
%%%%%%%%%%%%%%%%%%%%%%%%%%%%%%%%%%%%

%%%%%%%%%%%%%%%%%%%%%%%%%%%%%%%%%%%%
%%%%%%%%%%%%%%%%%%%%%%%%%%%%%%%%%%%%
%%%%%%%%%%%%%%%%%%%%%%%%%%%%%%%%%%%%

\section{Proof of Theorem \ref{Main2}}
\label{PROOF2}

We start with an improvement of \cite[Proposition 2.1]{GJ1}.
\begin{Lemma}
\label{Lemma:2}
For $\phi, \varphi \in \mathcal{D}(A)$, the rank one operator $|\phi \rangle \langle \varphi | : \psi \mapsto \langle \varphi, \psi \rangle \phi$ is of class $\CC^{1,\rm{u}}(A)$.  
\end{Lemma}

\begin{proof}
First, by \cite[Lemma 6.2.9]{ABG}, $|\phi \rangle \langle \varphi | \in \CC^1(A)$ if and only if the sesquilinear form  
\[ \D(A) \ni \psi \mapsto \langle \psi, [ | \phi \rangle \langle \varphi |, A ] \psi \rangle :=  \langle \langle \phi, \psi \rangle \varphi,  A \psi \rangle  - \langle A \psi,  \langle \varphi | \psi \rangle \phi \rangle \] 
is continuous for the topology induced by $\H$. Since 
\[ \langle \psi, [ | \phi \rangle \langle \varphi |, A ] \psi \rangle =  \langle \psi, \phi \rangle \langle A \varphi, \psi \rangle - \langle \psi, A \phi \rangle \langle \varphi, \psi \rangle = \langle \psi, \left(  |\phi \rangle \langle A \varphi |  - | A \phi \rangle \langle \varphi | \right) \psi \rangle, \]
we see that $|\phi \rangle \langle \varphi | \in \CC^1(A)$ and $[ | \phi \rangle \langle \varphi |, A ]_{\circ} =  |\phi \rangle \langle A \varphi |  - | A \phi \rangle \langle \varphi |$, which is a bounded operator of rank at most two. Apply Proposition \ref{PropC1UU}, more specifically Remark \ref{Remark2comp}, to obtain the result. \qed
\end{proof}
\vspace{0.5cm}
Next, we quote for convenience the result of \cite{Ri} that we use in the proof of Theorem \ref{Main2}.

\begin{theorem}\cite[Theorem 1]{Ri}
Let $H$ and $A$ be self-adjoint operators in $\H$ with $H \in \CC^{1,\rm u}(A)$. Assume that there exist an open interval $J \subset \R$ and $c >0$ such that $\eta (H) [H, \i A]_{\circ} \eta(H) \geqslant c \cdot \eta^2(H)$ for all real $\eta \in \CC_{\rm c} ^{\infty}(J)$. Let $a$ and $t$ be real numbers. Then for each real $\eta \in \CC_{\rm c}^{\infty}(J)$ and for each $v < c$ one has uniformly in $a$,
\[ \| E_{(-\infty, a+vt]}(A) e^{-\i tH} \eta(H) E_{[a,+\infty)}(A) \| \to 0 \quad \text{as} \quad t \to +\infty.\]
\end{theorem}

We are now ready to prove Theorem \ref{Main2}. \\

\vspace{0.2cm}

\noindent \textit{Proof of Theorem \ref{Main2}.}
Let $\I \subset \R$ be a compact interval as in the statement of Theorem \ref{Main2}, that is, for all $\lambda \in \I$, $\lambda \in \mu^A(H)$ and $\ker(H-\lambda) \subset \D(A)$. Let $\lambda \in \I$ be given, and assume that a Mourre estimate holds with $K \in \K(\H)$ in a neighborhood $J$ of $\lambda$.
\\ {\bf Step 1:} This step is a remark due to Serge Richard. In this step, we assume that $\lambda$ is not an eigenvalue of $H$. In this case, from the Mourre estimate, we may derive a strict Mourre estimate on a possibly smaller neighborhood of $\lambda$, because $E_{J}(H)KE_J(H)$ converges in norm to zero as the support of $J$ shrinks to zero around $\lambda$. So, without loss of generality, there is an open interval $J$ containing $\lambda$ and $c >0$ such that a strict Mourre estimate holds for $H$ on $J$, i.e.
\[ E_J(H) [H, \i A]_{\circ} E_J(H) \geqslant c E_J (H). \]
In particular, $J$ does not contain any eigenvalue of $H$. We look to apply \cite[Theorem 1]{Ri}. Let $\psi \in \H$, and assume without loss of generality that $\| \psi \| =1$.  Fix $v \in (0,c)$ and let $a \in \R$. Let $\eta \in \CC^{\infty}_c (J)$ be such that $\max_{x \in J}  | \eta (x) | \leqslant 1$, so that $\| \eta(H) \| \leqslant 1$. Note also that $\| \langle A \rangle ^{-s} \| \leqslant 1$ for all $s>0$. Then 
\begin{align*}
\| \langle A \rangle ^{-s} e^{-\i tH} \eta (H) \psi \| & \leqslant \| \langle A \rangle ^{-s} e^{-\i tH} \eta (H) E_{(-\infty, a)}(A) \psi \| + \| \langle A \rangle ^{-s} e^{-\i tH} \eta (H) E_{[a,+\infty)}(A) \psi  \| \\
& \leqslant \| E_{(-\infty, a)}(A) \psi \| + \| \langle A \rangle ^{-s} E_{(-\infty, a+vt]}(A) e^{-\i tH} \eta (H) E_{[a,+\infty)}(A)  \psi \| \\
& \quad + \| \langle A \rangle ^{-s} E_{(a+vt, +\infty)}(A) e^{-\i tH} \eta (H) E_{[a,+\infty)}(A)  \psi \| \\
& \leqslant \| E_{(-\infty, a)}(A) \psi \| +  \| E_{(-\infty, a+vt]}(A) e^{-\i tH} \eta (H) E_{[a,+\infty)}(A) \|  \\
& \quad + \| \langle A \rangle ^{-s} E_{(a+vt, +\infty)}(A) \| 
\end{align*}
Let $\epsilon > 0$ be given. Choose $a$ so that $\|E_{(-\infty, a)}(A) \psi \| \leqslant \epsilon/3$. Then take $t$ large enough so that the other two terms on the r.h.s.\ of the previous inequality are each less than $\epsilon/3$. The second one is controlled by \cite[Theorem 1]{Ri} and the third one by functional calculus. Then $\| \langle A \rangle ^{-s} e^{-\i tH} \eta (H) \psi \| \leqslant \epsilon$. Thus 
\[ \lim \limits_{t \to +\infty} \| \langle A \rangle ^{-s} e^{-\i tH} \eta (H) \psi \| = 0. \] 
By taking a sequence $\eta_k \in \CC^{\infty} _c(J)$ that converges pointwise to the characteristic function of $J$, we infer from the previous limit that
\[ \lim \limits_{t \to +\infty} \| \langle A \rangle ^{-s} e^{-\i tH} E_J (H) \psi \| = 0. \]
Finally, as there are no eigenvalues of $H$ in $J$, $E_{J}(H)=E_{J}(H) P_{\rm c}(H)$ and we have
\begin{equation}
\label{eqn1st} 
\lim \limits_{t \to +\infty} \| \langle A \rangle ^{-s} e^{-\i tH} E_J (H) P_{\rm c} (H) \psi \| = 0.
\end{equation}
\noindent {\bf Step 2:} In this step, $\lambda \in \I$ is assumed to be an eigenvalue of $H$. By adding a constant to $H$, we may assume that $\lambda \neq 0$. By assumption, there is an interval $J$ containing $\lambda$, $c >0$ and $K \in \K(\H)$ such that 
\[ E_J(H) [H, \i A]_{\circ} E_J(H) \geqslant c E_J(H) + K. \]
As the point spectrum of $H$ is finite in $J$, we further choose $J$ so that it contains only one eigenvalue of $H$, namely $\lambda$. Furthermore, the interval $J$ is chosen so that $0 \not \in J$. Denote $P = P_{\{ \lambda  \} }(H)$ and $P^{\perp} := 1 - P_{\{ \lambda \} }(H)$. Also let $H' := H P^{\perp}$. Then
\[ P^{\perp}E_J(H) [H, \i A]_{\circ} E_J(H) P^{\perp} \geqslant c E_J(H) P^{\perp} + P^{\perp} E_J (H) K E_J(H) P^{\perp}. \]
Functional calculus yields $P^{\perp} E_J (H) = E_J (H P ^{\perp})$ -- this is where the technical point $0 \not \in J$ is required. Moreover, $P^{\perp} E_J (H) K E_J (H) P^{\perp}$ converges in norm to zero as the size of the interval $J$ shrinks to zero around $\lambda$. Therefore there is $c' >0$ and an open interval $J'$ containing $\lambda$, with $J' \subset J$, such that 
\[ E_{J'}(H') [H', \i A]_{\circ} E_{J'}(H') \geqslant c' E_{J'}(H'). \]
In other words, a strict Mourre estimate holds for $H'$ on $J'$. Now $H' := H P^{\perp} = H - H P$. Note that $P$ is a finite sum of rank one projectors because $\lambda \in \mu^A(H)$. Thanks to the assumption $\ker(H-\lambda) \subset \D(A)$, we have by Lemma \ref{Lemma:2} that $P \in \CC^{1,\rm{u}}(A)$. Thus $H' \in \CC^{1,\rm{u}}(A)$. Performing the same calculation as in Step 1 with $(H',J')$ instead of $(H,J)$ gives 
\[ \lim \limits_{t \to + \infty} \| \langle A \rangle ^{-s} e^{-\i tH'} E_{J'} (H') \psi \| = 0. \] 
Since $e^{-\i tH'} E_{J'} (H') = e^{-\i tH} E_{J'} (H) P^{\perp}$, we have 
\[\lim \limits_{t \to + \infty} \| \langle A \rangle ^{-s} e^{-\i tH} E_{J'} (H) P^{\perp} \psi \| = 0. \]
The only eigenvalue of $H$ belonging to $J'$ is $\lambda$, so $E_{J'}(H) P^{\perp}=E_{J'}(H) P_{\rm c}(H)$. Thus
\begin{equation}
\label{eqn2nd} 
\lim \limits_{t \to + \infty} \| \langle A \rangle ^{-s} e^{-\i tH} E_{J'} (H) P_{\rm c}(H) \psi \| = 0.
\end{equation}
\noindent{\bf Step 3:} In this way, for each $\lambda \in \I$, we obtain an open interval $J_{\lambda}$ or $J'_{\lambda}$ containing $\lambda$ such that \eqref{eqn1st} or \eqref{eqn2nd} holds true, depending on whether $\lambda$ is an eigenvalue of $H$ or not. To conclude, as 
\[ \left( \bigcup \limits_{\substack{\lambda \in \I, \\ \lambda  \in \sigma_{\rm{c}}(H)}} J_{\lambda} \right) \bigcup  \left( \bigcup \limits_{\substack{\lambda \in \I, \\ \lambda \in \sigma_{\rm{pp}}(H)}}  J'_{\lambda} \right) \] 
is an open cover of $\I$, we may choose a finite sub-cover. If $\{J_i\}_{i=1}^n$ denotes this sub-cover, we may further shrink these intervals so that $J_i \cap J_j = \emptyset$ for $i \neq j$, $\overline{\cup J_i} = \I$, and $\overline{J_i} \cap \overline{J_j} \in \sigma_{\rm{c}}(H)$ for $\i \neq j$. Thus $E_{\I}(H) P_{\rm{c}}(H) = \sum_i ^n E_{J_i}(H) P_{\rm{c}}(H)$. Then, by applying \eqref{eqn1st} and \eqref{eqn2nd} we get
\[ \lim \limits_{t \to + \infty} \| \langle A \rangle ^{-s} e^{-\i tH} E_{\I}(H) P_{\rm c}(H) \psi \| \leqslant \lim \limits_{t \to + \infty} \sum_{i=1}^n \| \langle A \rangle ^{-s} e^{-\i tH} E_{J_i}(H) P_{\rm c}(H) \psi \| = 0. \]
This proves the estimate \eqref{NewFormula3}.
\\
\noindent{\bf Step 4:} We turn to the proof of \eqref{NewFormula4}. Since, $A$ is self-adjoint, $\D(A)$ is dense in $\H$. Let $\{ \phi_n \}_{n=1}^{\infty} \subset \D(A)$ be an orthonormal set. Let $W \in \K(\H)$ and denote $F_N := \sum _{n=1} ^N \langle \phi_n, \cdot \rangle W\phi_n$. The proof of \cite[Theorem VI.13]{RS1} shows that $\|  W-F_N \| \to 0$ as $N \to + \infty$. Then
\begin{align*}
 \| W P_{\rm{c}} (H) E_\I (H)  e^{-\i tH} \psi \| &\leqslant  \| (W-F_N) P_{\rm{c}} (H) E_\I (H)  e^{-\i tH} \psi \| +  \| F_N P_{\rm{c}} (H) E_\I (H)  e^{-\i tH} \psi \| \\
& \leqslant \underbrace{\| W-F_N \|}_{\to \ 0 \ \text{as} \ N \to +\infty} +  \ \| F_N \langle A \rangle \| \underbrace{\| \langle A \rangle ^{-1} P_{\rm{c}} (H) E_\I (H)  e^{-\i tH} \psi \|}_{\to \ 0 \ \text{as} \ t \to + \infty}.  
\end{align*}
The result follows by noting that $F_N \langle A \rangle$ is a bounded operator for each $N$. If $W$ is $H$-relatively compact,  use the fact that $E_\I(H) (H+\i)$ is a bounded operator.
\qed

%%%%%%%%%%%%%%%%%%%%%%%%%%%%%%%%%%%%
%%%%%%%%%%%%%%%%%%%%%%%%%%%%%%%%%%%%
%%%%%%%%%%%%%%%%%%%%%%%%%%%%%%%%%%%%
%%%%%%%%%%%%%%%%%%%%%%%%%%%%%%%%%%%%
%%%%%%%%%%%%%%%%%%%%%%%%%%%%%%%%%%%%
%%%%%%%%%%%%%%%%%%%%%%%%%%%%%%%%%%%%

%%%%%%%%%%%%%%%%%%%%%%%%%%%%%%%%%%%%
%%%%%%%%%%%%%%%%%%%%%%%%%%%%%%%%%%%%   <------
%%%%%%%%%%%%%%%%%%%%%%%%%%%%%%%%%%%%

\section{Proof of Theorem \ref{Main}}
\label{PROOF}

To prove the result, we will need the following fact:
%\begin{Lemma}
%\label{Lemma:1}
%\cite{GJ1} For $\psi, \varphi \in \mathcal{D}(A)$, the rank one operator $|\phi \rangle \langle \varphi | : \psi \mapsto \langle \varphi, \psi \rangle \phi$ is of class $\CC^1(A)$.  
%\end{Lemma}
\begin{Lemma}
\label{Lem1}
Let $T$ be a self-adjoint operator with $T \in \CC^1(A)$. Let $\lambda \in \mu^A(T)$ and suppose that $\ker(T-\lambda) \subset \mathcal{D}(A)$. Then there is an interval $\I \subset \mu^A(T)$ containing $\lambda$ such that $P_{\rm{c}}^{\perp}(T) E_{\I}(T)$ and $P_{\rm{c}} (T) \eta(T)$ are of class $\CC^1(A)$ for all $\eta \in C_c^{\infty}(\R)$ with $\text{supp}(\eta) \subset \I$. 
\end{Lemma}
\begin{proof}
Since there are finitely many eigenvalues of $T$ in a neighborhood of $\lambda$, there is a bounded interval $\I$ containing $\lambda$ such that $\ker(T-\lambda') \subset \D(A)$ for all $\lambda' \in \I$. Then $P_{\rm{c}}^{\perp}(T) E_{\I}(T)$ is a finite rank operator and belongs to the class $\CC^1(A)$ by Lemma \ref{Lemma:2}. Moreover $T \in \CC^1(A)$ implies $\eta(T) \in \CC^1(A)$, by the Helffer-Sj\"ostrand formula. So $P_{\rm{c}}^{\perp} (T) E_{\I}(T) \eta(T) \in \CC^1(A)$ as the product of two bounded operators in this class. Finally, $P_{\rm{c}} (T) \eta(T) = \eta(T) - P_{\rm{c}}^{\perp} (T) E_{\I}(T) \eta(T)$ is a difference of two bounded operators in $\CC^1(A)$, so $P_{\rm{c}} (T) \eta(T) \in \CC^1(A)$. 
\qed
\end{proof}

\vspace{0.5cm}
\noindent \textit{Proof of Theorem \ref{Main}.}
Since $H_0$ is semi-bounded and $\sigma_{\text{ess}}(H) = \sigma_{\text{ess}}(H_0)$, there is $\varsigma \in \R \setminus (\sigma(H) \cup \sigma(H_0))$. Denote the resolvents of $H$ and $H_0$ respectively by $R(z) := (z-H)^{-1}$ and $R_0(z) := (z-H_0)^{-1}$. Also denote the spectral projector of $R(z)$ onto the continuous spectrum by $P_{\rm{c}}(R(z))$. We split the proof into four parts. First we translate the problem in terms of the resolvent $R(\varsigma)$. Second we show the following formula: 
\begin{equation}
\label{Formula1}
\begin{split}
& P_{\rm{c}}(R(\varsigma)) \theta(R(\varsigma)) [R(\varsigma), \i \varphi(A/L)]_{\circ} \theta(R(\varsigma))  P_{\rm{c}}(R(\varsigma))  \geqslant \\
& \ \ \ \ \ \ \ \ \ \ \ \ \ \ \ \ \ \ \ \ \ \ \ \ \ \  L^{-1} P_{\rm{c}}(R(\varsigma)) \theta(R(\varsigma)) \Big( C \langle A/L \rangle ^{-2s} + K \Big) \theta(R(\varsigma)) P_{\rm{c}}(R(\varsigma)),
\end{split}
\end{equation}
where $\theta$ is a smooth function compactly supported about $(\varsigma - \lambda)^{-1}$, $\varphi$ is an appropriately chosen smooth bounded function, $L \in \R^+$ is sufficiently large, $K$ is a compact operator uniformly bounded in $L$, $C>0$, and $s \in (1/2,1)$. $\theta, \varphi, C$ and $s$ are independent of $L$. This formula is expressed in terms of the resolvent $R(\varsigma)$. Third, we look to convert it into a formula for $H$. We show that the latter formula implies the existence of an open interval $J$ containing $\lambda$ such that
\begin{equation}
\begin{split}
\label{Formula2}
&    P_{\rm{c}}(H) E_{J}(H) [R(\varsigma), \i \varphi(A/L)]_{\circ}  E_{J}(H) P_{\rm{c}}(H) \geqslant \\
& \ \ \ \ \ \ \ \ \ \ \ \ \ \ \ \ \ \ \ L^{-1} P_{\rm{c}}(H) E_{J}(H) \Big( C \langle A/L \rangle ^{-2s} + K \Big) E_{J}(H) P_{\rm{c}}(H).
\end{split}
\end{equation}
We note that the operator $K$ is the same in \eqref{Formula1} and \eqref{Formula2}. Fourth, we insert the dynamics into the previous formula and average over time. We notably use the RAGE Theorem \eqref{RAGEsup} to derive the desired formula, i.e.
\begin{equation}
\label{Formula3}
\lim \limits_{T \to \pm \infty} \sup \limits_{\| \psi\| \leqslant 1} \frac{1}{T} \int_0 ^T  \| \langle A \rangle ^{-s} e^{-\i t H} P_{\rm{c}}(H) E_{J}(H)\psi \|^2 dt = 0.
\end{equation}

\noindent \textbf{Part 1:}  Let $\lambda \in \mu^A(H)$ be such that $\ker(H-\lambda) \subset \mathcal{D}(A)$. Then there are finitely many eigenvalues in a neighborhood of $\lambda$ including multiplicity. We may find an interval $\I = (\lambda_0,\lambda_1)$ containing $\lambda$ such that $\I \subset \mu^A(H)$ and for all $\lambda' \in \I$, $\ker(H-\lambda') \subset \mathcal{D}(A)$. Define 
\begin{equation}
\label{Functionf}
f:\R \setminus \{\varsigma\} \mapsto \R,  \quad f:x \mapsto 1/(\varsigma-x).
\end{equation}
Since eigenvalues of $H$ located in $\I$ are in one-to-one correspondence with the eigenvalues of $R(\varsigma)$ located in $f(\I) = (f(\lambda_0),f(\lambda_1))$, it follows that $f(\I)$ is an interval containing $f(\lambda)$ such that $f(\I) \subset \mu^A(R(\varsigma))$ and $\ker(R(\varsigma)-\lambda') \subset \mathcal{D}(A)$ for all $\lambda' \in f(\I)$. Note the use of Proposition \ref{conjugateResolvent}.

To simplify the notation in what follows, we let $R := R(\varsigma)$, $R_0 := R_0(\varsigma)$ and $P_{\rm{c}} := P_{\rm{c}}(R(\varsigma))$, as $\varsigma$ is fixed. Also let $R_A(z) := (z-A/L)^{-1}$, where $L \in \R^+$.

\noindent \textbf{Part 2:} Let $\theta,\eta,\chi \in C^{\infty}_c (\R)$ be bump functions such that $f(\lambda) \in$ supp$(\theta) \subset$ supp$(\eta) \subset$ supp$(\chi) \subset f(\I)$, $\eta \theta = \theta$ and $\chi \eta =\eta$. Let $s \in (1/2,1)$ be given. Define 
\begin{equation*}
\varphi : \R \mapsto \R, \qquad \varphi : t \mapsto \int _{-\infty} ^t \langle x \rangle ^{-2s} dx.
\end{equation*}
Note that $\varphi \in \mathcal{S}^{0}(\R)$. The definition of $\mathcal{S}^0(\R)$ is given in \eqref{decay1}. Consider the bounded operator
\begin{equation*}
F  := \Pp \tH [\RH, \i \varphi(A/L)]_{\circ} \tH \Pp = \frac{\i}{2\pi L} \int_{\C} \pp \Pp \tH  R_A(z) [\RH, \i A]_{\circ} R_A(z) \tH \Pp  \ \dz.
\end{equation*}
By Lemma \ref{Lem1} with $T = R$, $\Pp \eH \in C^1(A)$, so 
\begin{equation*}
[\Pp \eH,R_A(z)]_{\circ} = L^{-1} R_A(z) [\Pp \eH, A]_{\circ} R_A(z).
\end{equation*} 
In the formula defining $F$, we introduce $\Pp\eH$ next to $\Pp \tH$ and commute it with $R_A(z)$:
\begin{align*}
\label{la;}
F &= \frac{\i}{2\pi L} \int_{\C} \pp \Pp \tH  \Big(R_A(z) \Pp \eH + [\Pp \eH,R_A(z)]_{\circ} \Big) [\RH, \i A]_{\circ} \times \\ 
& \ \ \ \ \ \ \ \ \ \ \ \ \ \ \ \ \ \ \ \ \ \ \ \ \Big(\eH \Pp R_A(z)+ [R_A(z),\Pp \eH]_{\circ} \Big)  \tH \Pp \ \dz \\
&= \frac{\i}{2\pi L} \int_{\C} \pp \Pp \tH  R_A(z) \Pp \eH [\RH, \i A]_{\circ} \eH \Pp R_A(z) \tH \Pp \ \dz   \\
&\quad + L^{-1} \Pp \tH  \left( I_1 + I_2 + I_3 \right)  \tH   \Pp,
\end{align*}
where 
\begin{align*}
I_1 &= \frac{\i}{2\pi} \int _{\C} \pp   [\Pp\eH, R_A(z)]_{\circ} [\RH, \i A]_{\circ} \eH \Pp R_A(z) \ \dz, \\
I_2 &= \frac{\i}{2\pi} \int _{\C} \pp   [\Pp\eH, R_A(z)]_{\circ} [\RH, \i A]_{\circ}   \ \dz, \\
I_3 &= \frac{\i}{2\pi} \int _{\C} \pp   R_A(z) \Pp  \eH [\RH, \i A]_{\circ} [R_A(z), \Pp\eH]_{\circ} \ \dz. 
\end{align*}
Applying \eqref{dei} and Lemma \ref{didid2}, and recalling that $s<1$, we have for some operators $B_i$ uniformly bounded with respect to $L$ that
\begin{equation*}
I_i = \JapA ^{-s} \frac{B_i}{L} \JapA ^{-s}, \quad \text{for} \ i=1,2,3.
\end{equation*}
Using $\chi \eta = \eta$, we insert $\chiH$ next to $\eH$ . So far we get the following expression for $F$ :
\begin{align*}
F &= \frac{\i}{2\pi L} \int_{\C} \pp \Pp \tH  R_A(z) \Pp \chiH \underbrace{\eH [\RH, \i A]_{\circ} \eH}_{\text{to be developed}} \chiH \Pp R_A(z) \tH \Pp \ \dz \\
&\quad + \Pp \tH \JapA ^{-s} \left(\frac{B_1+B_2+B_3}{L^2} \right) \JapA ^{-s} \tH \Pp.
\end{align*}
Now write 
\begin{equation*}
\eH [\RH, \i A]_{\circ} \eH = \eH R [H, \i A]_{\circ} R \eH = \eH R [H_0, \i A]_{\circ} R \eH + \eH R [V, \i A]_{\circ} R \eH.
\end{equation*}
Let us start with the second term on the r.h.s.\ of this equation. It decomposes into  
\begin{align*}
\eH R [V , \i A]_{\circ} R \eH &= \eH \underbrace{R \langle H \rangle}_{\in \ \B(\H)} \langle H \rangle ^{-1/2}  \underbrace{\langle H \rangle ^{-1/2} \langle H_0 \rangle^{1/2}}_{\in \ \B(\H)} \underbrace{\langle H_0 \rangle ^{-1/2} [V , \i A]_{\circ} \langle H_0 \rangle ^{-1/2} }_{\in \ \K(\H) \ \text{by \ref{item:A6prime}}} \times \\
& \quad \times \underbrace{ \langle H_0 \rangle ^{1/2}\langle H \rangle ^{-1/2}}_{\in \ \B(\H)} \langle H \rangle ^{-1/2} \underbrace{\langle H \rangle R}_{\in \ \B(\H)}  \eH.
 \end{align*}
 It is therefore compact. As for the first term on the r.h.s., it decomposes as follows
\begin{equation*}
\eH R [H_0, \i A]_{\circ} R \eH = \eHz R_0 [H_0, \i A]_{\circ} R_0 \eHz + \Xi_1 + \Xi_2,
\end{equation*}
where 
\begin{equation*}
\Xi_1 := (\eH R -\eHz R_0) [H_0, \i A]_{\circ} R \eH \quad \text{and} \quad \Xi_2 := \eHz R_0 [H_0, \i A]_{\circ} (R \eH- R_0\eHz).
\end{equation*}
We show tht $\Xi_1$ is compact, and similarly one shows that $\Xi_2$ is compact. We have
\begin{equation*}
\Xi_1= \underbrace{(\eH R -\eHz R_0) \langle H_0 \rangle ^{1/2}}_{\in \ \K(\H)} \underbrace{\langle H_0 \rangle ^{-1/2} [H_0 , \i A]_{\circ} \langle H_0 \rangle ^{-1/2} }_{\in \ \B(\H) \ \text{by \ref{item:A2}}} \underbrace{ \langle H_0 \rangle ^{1/2}\langle H \rangle ^{-1/2}}_{\in \ \B(H)} \langle H \rangle ^{-1/2} \underbrace{\langle H \rangle R}_{\in \ \B(\H)} \eH.
\end{equation*}
Let us justify that $(\eH R -\eHz R_0) \langle H_0 \rangle ^{1/2}$ is compact. Let $ \kappa : x \mapsto x \eta(x)$. By the Helffer-Sj\"ostrand formula,
\begin{align*}
(\eH R -\eHz R_0) \langle H_0 \rangle ^{1/2} &=  \frac{\i}{2\pi} \int_{\C} \frac{\partial \tilde{\kappa}}{\partial \overline{z}} (z) \left( (z-R)^{-1} - (z-R_0)^{-1} \right) \langle H_0 \rangle ^{1/2} \ d z \wedge d\overline{z} \\
&=  \frac{\i}{2\pi} \int_{\C} \frac{\partial \tilde{\kappa}}{\partial \overline{z}} (z) (z-R)^{-1} RVR_0 (z-R_0)^{-1}  \langle H_0 \rangle ^{1/2} \ d z \wedge d\overline{z} \\
&= \frac{\i}{2\pi} \int_{\C} \frac{\partial \tilde{\kappa}}{\partial \overline{z}} (z)  (z-R)^{-1} \underbrace{R \langle H \rangle^{1/2}}_{\in \ \B(\H)}  \underbrace{\langle H \rangle ^{-1/2} \langle H_0 \rangle^{1/2}}_{\in \ \B(H)}  \times \\
& \quad \times \underbrace{\langle H_0 \rangle^{-1/2} V\langle H_0 \rangle^{-1/2}}_{\in \ \K(\H) \ \text{by \ref{item:A5}}} \underbrace{\langle H_0 \rangle^{1/2} R_0 \langle H_0 \rangle ^{1/2}}_{\in \ \B(\H)} (z-R_0)^{-1} \ d z \wedge d\overline{z}.
\end{align*}
The integrand of this integral is compact for all $z \in \C \setminus \R$, and moreover the integral converges in norm since $\kappa$ has compact support. It follows that $(\eH R -\eHz R_0) \langle H_0 \rangle ^{1/2}$, and thus $\Xi_1$, is compact. Thus we have shown that 
\begin{equation}
\label{KeyFormula1}
\eH [\RH, \i A]_{\circ} \eH = \eHz  [R_0, \i A]_{\circ}  \eHz + \text{compact}.
\end{equation}
Therefore there is a compact operator $K_1$ uniformly bounded in $L$ such that
\begin{align*}
F &= \constL \int_{\C} \pp \Pp \tH R_A(z)  M R_A(z) \tH \Pp \ \dz \\
&\quad +\Pp \tH \frac{K_1}{L}  \tH \Pp   \  + \  \Pp \tH \JapA ^{-s} \left(\frac{B_1+B_2+B_3}{L^2} \right) \JapA ^{-s} \tH \Pp.
\end{align*}
Here $M := \Pp \chiH \eHz [\RHz, \i A]_{\circ} \eHz \chiH \Pp$. Since $\Pp \chiH, \eta(R_0)$ and $[R_0,\i A]_{\circ}$ belong to $\CC^1(A)$, it follows by product that $M \in \CC^1(A)$ and we may commute $R_A(z)$ with $M$:
\begin{align*}
F &= \constL \int_{\C} \pp \Pp \tH R_A(z)^{2} M \tH \Pp \ \dz \\
&\quad+ \constL \int_{\C} \pp \Pp \tH R_A(z) [M, R_A(z)]_{\circ} \tH \Pp \ \dz \\
&\quad+\Pp \tH \frac{K_1}{L}  \tH \Pp   \  + \  \Pp \tH \JapA ^{-s} \left(\frac{B_1+B_2+B_3}{L^2} \right) \JapA ^{-s} \tH \Pp.
\end{align*}
We apply \eqref{derivative} to the first integral (which converges in norm), while for the second integral we use the fact that $M \in \CC^1(A)$ to conclude that there exists an operator $B_4$ uniformly bounded in $L$ such that 
\begin{align*}
F &= L^{-1} \Pp \tH \varphi'(A/L) M \tH \Pp \\
&\quad+\Pp \tH \frac{K_1}{L}  \tH \Pp   \  + \  \Pp \tH \JapA ^{-s} \left(\frac{B_1+B_2+B_3+B_4}{L^2} \right) \JapA ^{-s} \tH \Pp.
\end{align*}
Now $\varphi'(A/L) = \langle A/L \rangle ^{-2s}$. As a result of the Helffer-Sj\"{o}strand formula, \eqref{dei} and \eqref{use2},
\begin{equation*}
[\langle A/L \rangle ^{-s}, M ]_{\circ} \langle A/L \rangle ^{s} = L^{-1}B_5
\end{equation*} 
for some operator $B_5$ uniformly bounded in $L$. Thus commuting $\langle A/L \rangle ^{-s}$ and $M$ gives
\begin{align*}
F &= L^{-1} \Pp \tH \JapA ^{-s} M \JapA ^{-s} \tH \Pp \\
&\quad +\Pp \tH \frac{K_1}{L}  \tH \Pp   \  + \  \Pp \tH \JapA ^{-s} \left(\frac{B_1+B_2+B_3+B_4+B_5}{L^2} \right) \JapA ^{-s} \tH \Pp \\
& \geqslant c L^{-1} \Pp \tH \JapA ^{-s} \Pp \chiH \eHz^2 \chiH \Pp \JapA ^{-s} \tH \Pp  \\
&\quad +\Pp \tH \frac{K_1 + K_2}{L}  \tH \Pp   \  + \  \Pp \tH \JapA ^{-s} \left(\frac{B_1+B_2+B_3+B_4+B_5}{L^2} \right) \JapA ^{-s} \tH \Pp,
\end{align*}
where $c>0$ and $K_2$ come from applying the Mourre estimate \eqref{MourreEst} to $R_0$ on $f(\I)$. Exchanging $\eHz^2$ for $\eH^2$, we have a compact operator $K_3$ uniformly bounded in $L$ such that  
\begin{align*}
F & \geqslant c L^{-1} \Pp \tH \JapA ^{-s} \Pp \chiH \eH^2  \chiH \Pp \JapA ^{-s} \tH \Pp  \ + \ \Pp \tH \frac{K_1 +K_2+K_3}{L}  \tH \Pp  \\
&\quad+ \Pp \tH \JapA ^{-s} \left(\frac{B_1+B_2+B_3+B_4+B_5}{L^2} \right) \JapA ^{-s} \tH \Pp.
\end{align*}
We commute $\Pp \chiH \eH^2 \chiH \Pp = \Pp \eH^2 \Pp$ with $\langle A/L \rangle ^{-s}$, and see that 
\begin{equation*}
[\Pp \eH^2 \Pp,\langle A/L \rangle ^{-s}]_{\circ} \langle A/L \rangle ^{s} = L^{-1}B_6
\end{equation*}
for some operator $B_6$ uniformly bounded in $L$. Thus 
\begin{align*}
F & \geqslant c L^{-1} \Pp \tH \JapA ^{-2s} \tH \Pp \ + \ \Pp \tH \frac{K_1 +K_2+K_3}{L}  \tH \Pp \\
&  \quad +   \Pp \tH \JapA ^{-s} \left(\frac{B_1+B_2+B_3+B_4+B_5+B_6}{L^2} \right) \JapA ^{-s} \tH \Pp.
\end{align*}
Taking $L$ large enough gives $C>0$ such that $c + (B_1+B_2+B_3+B_4+B_5+B_6)/L \geqslant C$. Denoting $K := K_1+K_2+K_3$ yields formula \eqref{Formula1}.

\noindent \textbf{Part 3:} For all open intervals $(e_1,e_2)$ located strictly above or below $\varsigma$ we have the identity 
\begin{equation}
\label{ResolventProjector}
E_{(e_1,e_2)}(H) = E_{(f(e_1),f(e_2))}(R(\varsigma)),
\end{equation}
where $f$ is the function defined in \eqref{Functionf}. Now let $\mathcal{J} :=$ interior$(\theta^{-1}\{1\})$. This is an open interval and we have $E_{\mathcal{J}}(R) \theta(R) = E_\mathcal{J}(R)$. Thus applying $E_\mathcal{J}(R)$ to \eqref{Formula1} gives
\begin{equation*}
\Pp E_\mathcal{J}(R) [\RH, \i \varphi(A/L)]_{\circ} E_\mathcal{J}(R) \Pp \geqslant L^{-1} \Pp E_\mathcal{J}(R) \left(C \langle A/L \rangle ^{-2s} +K\right) E_\mathcal{J}(R) \Pp.
\end{equation*}
We have that $P_{\rm{c}} E_\mathcal{J}(R) := P_{\rm{c}}(R) E_\mathcal{J}(R)$ is a spectral projector of $R$ onto a finite disjoint union of open intervals. Let $\{\lambda_i\}$ be the (finite) collection of eigenvalues of $R$ located in $\mathcal{J}$. Then $\{f^{-1}(\lambda_i)\}$ are the eigenvalues of $H$ located in $f^{-1}(\mathcal{J})$, and by \eqref{ResolventProjector},
\begin{equation*}
P_{\rm{c}}(R) E_\mathcal{J}(R) = \sum_i E_{\mathcal{J}_i}(R) = \sum_i E_{f^{-1}(\mathcal{J}_i)}(H) = P_{\rm{c}}(H) E_{f^{-1}(\mathcal{J})}(H),
\end{equation*}
where the $\mathcal{J}_i$ are the open intervals such that $\cup_i \mathcal{J}_i \cup \{ \lambda_i\} = \mathcal{J}$. Denoting the open interval $J := f^{-1}(\mathcal{J})$ proves formula \eqref{Formula2}. Note that $\lambda \in J$.

\noindent \textbf{Part 4:} Let $F'$ be the l.h.s.\ of \eqref{Formula2}, i.e.\
\begin{equation*}
F' := P_{\rm{c}}(H) E_{J}(H) [R(\varsigma), \i \varphi(A/L)]_{\circ} E_{J}(H)  P_{\rm{c}}(H). 
\end{equation*}

\noindent Formula \eqref{Formula2} implies that for all $\psi \in \mathcal{H}$ and all $T >0$: 
\begin{align*}
\frac{L}{T} \int_0 ^T \langle e^{-\i tH} \psi, F' e^{-\i t H} \psi \rangle dt & \geqslant \frac{C}{T} \int_0 ^T \Big \| \langle A/ L \rangle ^{-s} E_{J}(H)  P_{\rm{c}}(H)   e^{-\i t H} \psi \Big \|^2 \ dt  + \\
&\quad+ \frac{1}{T} \int_0 ^T \langle E_{J}(H)  P_{\rm{c}}(H)   e^{-\i tH} \psi, K  E_{J}(H) P_{\rm{c}}(H)  e^{-\i t H} \psi \rangle \ dt.
\end{align*}
First, for all $L \geqslant 1$, 
\begin{equation*}
\frac{L}{T} \int_0 ^T e^{\i tH} F' e^{-\i t H} \ dt = \frac{L}{T} \big [ e^{\i tH}P_{\rm{c}}(H) E_{J}(H) R(\varsigma) \varphi(A/L) R(\varsigma) E_{J}(H)  P_{\rm{c}}(H) e^{-\i t H} \big]_0 ^T \xrightarrow[T \to \pm \infty]{} 0.
\end{equation*}
Second, by the RAGE Theorem \eqref{RAGEsup},
\begin{equation*}
\begin{split}
\sup \limits_{\| \psi\| \leqslant 1} \frac{1}{T} \int_0 ^T  \langle &  E_{J}(H)  P_{\rm{c}}(H)  e^{-\i tH} \psi,  K  E_{J}(H) P_{\rm{c}}(H) e^{-\i t H} \psi \rangle \ dt  \\
& \leqslant  \sup \limits_{\| \psi\| \leqslant 1} \frac{1}{T} \int_0 ^T \| K  E_{J}(H)  e^{-\i t H} P_{\rm{c}}(H) \psi \| \ dt \\
& \leqslant \sup \limits_{\| \psi\| \leqslant 1} \left(\frac{1}{T} \int_0 ^T \| K  E_{J}(H)  e^{-\i t H} P_{\rm{c}}(H) \psi \|^2 \ dt \right)^{1/2} \xrightarrow[T \to \pm \infty]{} 0.
\end{split}
\end{equation*}
It follows that for $L$ sufficiently large (but finite),
\begin{equation*}
\lim \limits_{T \to \pm \infty} \sup \limits_{\| \psi\| \leqslant 1} \frac{1}{T} \int_0 ^T  \bigg \| \JapA ^{-s}  e^{-\i t H}  P_{\rm{c}}(H)  E_{J}(H)  \psi \bigg \|^2 \ dt = 0.
\end{equation*}
Finally \eqref{Formula3} follows by noting that $\langle A \rangle^{-s} \langle A / L \rangle ^{s}$ is a bounded operator.  \qed

\section{A discussion about the compactness of operators of the form $\langle A \rangle ^{-s} E_{\I}(H)$}
\label{Section:Compactness}

As pointed out in the Introduction, the novelty of formula \eqref{NewFormula} is conditional on the non-relative compactness of the operator $\langle A \rangle ^{-s} E_{\I}(H)$. The non-compactness of $\langle A \rangle ^{-s} E_{\I}(H)$ is also what sets \eqref{NewFormula3} apart from \eqref{NewFormula4}. We start by noting that $\langle A \rangle ^{-s} E_{\I}(H)$ is $H$-relatively compact if and only if it is compact, since $\I \subset \R$ is a bounded interval. 

We will allow ourselves to consider operators of the form $\langle A \rangle ^{-s} \chi(H)$, where $\chi$ is a smooth function, rather than $\langle A \rangle ^{-s} E_{\I}(H)$. On the one hand, if $\langle A \rangle ^{-s} E_{\I}(H)$ is compact, then so is $\langle A \rangle ^{-s} \chi(H)$, where $\chi$ is any smooth function that has support contained in $\I$. On the other hand, if $\langle A \rangle ^{-s} \chi(H)$ is compact, where $\chi$ is a smooth bump function that approximates the characteristic function of $\I$ and equals one above $\I$, then so is $\langle A \rangle ^{-s} E_{\I}(H)$. 

We will also suppose that $H = H_0 + V$, where $V$ is some $H_0$-form compact operator, and $H_0$ is viewed as the "free" operator. In other words we will work under the assumption \ref{item:A5}.
The reason for doing so is that $H_0$ is much easier to work with than $H$ in practice. In this case we note that $\langle A \rangle ^{-s} \chi(H)$ is compact if and only if $\langle A \rangle ^{-s} \chi(H_0)$. We therefore have the question: Is $\langle A \rangle ^{-s} \chi(H_0)$ a compact operator? A first result is:\begin{proposition}
\label{propNoEigenvalues}
Let $H_0,A$ be self-adjoint operators in $\H$. Suppose that $H_0$ has a spectral gap. Suppose that $H_0 \in \CC^1(A)$ and that for some $\lambda \in \R$, $[(H_0-\lambda)^{-1},\i A]_{\circ} := C \geqslant 0$ is an injective operator. Then $A$ does not have any eigenvalues. In particular, $\langle A \rangle ^{-s} \not \in \K(\H)$ for any $s>0$.
\end{proposition}
\begin{remark}
The examples of Section \ref{Section:Examples} satisfy the hypotheses of Proposition \ref{propNoEigenvalues}. The positivity of $C$ holds because $\sigma(H_0) \subset [0,+\infty)$. The injectivity holds because $0$ is not an eigenvalue of $H_0$. 
\end{remark}
\begin{proof}
Let $\psi$ be an eigenvector of $A$. Since $A \in \CC^1((H_0-\lambda)^{-1})$, the Virial Theorem (\cite[Proposition 7.2.10]{ABG}) says that $0= \langle \psi, [(H_0-\lambda)^{-1},\i A]_{\circ} \psi \rangle = \langle \psi, C \psi \rangle = \|\sqrt{C} \psi \|^2$. The injectivity of $\sqrt{C}$ forces $\psi = 0$, i.e.\ $\sigma_{\rm{p}}(A) = \emptyset$. Now, it is known that the spectrum of a self-adjoint operator with compact resolvent consists solely of isolated eigenvalues of finite multiplicity, see e.g.\ \cite[Theorem 6.29]{K}. So if $A$ had compact resolvent, then we would have $\sigma(A) = \sigma_{\rm{p}}(A) = \emptyset$. However this is not possible because the spectrum of a self-adjoint operator is non-empty. We conclude that $A$ does not have compact resolvent. Writing $(z-A)^{-1} = (z-A)^{-1} \langle A \rangle \langle A \rangle ^{-1}$, we infer that $\langle A \rangle ^{-1} \not \in \K(\H)$. Finally, consider the bounded self-adjoint operator $\langle A \rangle ^{-s}$ for some $s>0$. If this operator were compact, then by the spectral theorem for such operators we would have $\langle A \rangle ^{-s} = \sum_i \lambda_i \langle \phi_i, \cdot \rangle \phi_i$ for some eigenvalues $\{\lambda_i\}$ and eigenvectors $\{\phi_i\}$ which form an orthonormal basis of $\H$. But then $\langle A \rangle ^{-1} = \sum_i \lambda_i^{1/s} \langle \phi_i, \cdot \rangle \phi_i$, implying that the latter operator is compact. This contradiction proves $\langle A \rangle ^{-s} \not \in \K(\H)$ for all $s>0$. \qed
\end{proof}
Unfortunately, this result does not settle the debate because it does not guarantee the non-compactness of $\langle A \rangle ^{-s} \chi(H_0)$. In fact, we have examples where this operator is compact. For lack of a more robust result, we shall spend the rest of this section examining several examples. Our conclusion is that $\langle A \rangle ^{-s} \chi(H_0)$ is sometimes compact, sometimes not. Specifically, in each of our examples, the compactness holds in dimension one but does not in higher dimensions. To start off, we cook up a simple example that will reinforce the viewpoint that non-compactness is possible, especially in higher dimensions.

\begin{example} 
%\normalfont
Let $\H := L^2(\R^2)$, $H_0:= -\partial^2/ \partial x_1^2$ and $A := -\i(x_1 \partial / \partial x_1 + \partial / \partial x_1 x_1)$ be a conjugate operator to $H_0$. The spectrum of $H_0$ is purely absolutely continuous and $\sigma(H_0) = [0,+\infty)$. In particular, $\ker(H_0-\lambda) = \emptyset$ for all $\lambda \in \R$. Also $[H_0, \i A]_{\circ} = 2H_0$ exists as a bounded operator from $\D(\langle H_0 \rangle ^{1/2})$ to $\D(\langle H_0 \rangle ^{1/2})^*$, implying that $H_0 \in \CC^{\infty}(A)$ and that the Mourre estimate holds for all positive intervals supported away from zero. In addition, $\{e^{\i tA}\}_{t \in \R}$ stabilizes $\D(H_0)$. The assumptions of Theorems \ref{Main2} and \ref{Main} are therefore thoroughly verified. Moreover $\langle A \rangle ^{-s} \chi (H_0)$ is clearly not compact in $L^2(\R^2)$. This can be seen by applying $\langle A \rangle ^{-s} \chi (H_0)$ to a sequence of functions $(f(x_1) g_n (x_2))_{n=1} ^{+\infty}$ with $g_n$ chosen so that $\int_{\R} |g_n(x_2)|^2 dx_2$ is constant.  

\end{example}

To continue with other examples, we set up notation. Let $\CC_0(\R)$ be the continuous functions vanishing at infinity and $\CC_c^{\infty}(\R)$ the compactly supported smooth functions.

\begin{example}
\label{ex:Derivative}
%\normalfont
Let $\mathcal{H} := L^2(\R^d)$, $H_0 := x_1+...+x_d$ and $A:= \i(\partial / \partial x_1 +...+ \partial / \partial x_d)$. This system verifies the Mourre estimate at all energies thanks to commutator relation $[H_0,\i A]_{\circ} = d I$, and $H_0 \in \CC^{\infty}(A)$ holds. Although this system does not quite fall within the framework of this article because $H_0$ is not semi-bounded ($\sigma(H_0) =\R$), it conveys the idea that compactness holds only in dimension one: 
\end{example}

\begin{proposition}
\label{TwoDimensionKnownResult}
Let $H_0$ and $A$ be those from Example \ref{ex:Derivative}. Let $\chi \in \CC_0(\R)$ and $s \in \R$ be given. If $d=1$, then $\langle A \rangle ^{-s} \chi(H_0) \in \K(L^2(\R))$. If $d=2$, then $\langle A \rangle ^{-s} \chi(H_0) \not\in \K(L^2(\R^2))$.
\end{proposition}
\begin{proof} The one-dimensional result is a classic, see Proposition \ref{KnownFourierDecay}. We prove the two-dimensional result. Let $\I(\lambda,r)$ denote the open interval centered at $\lambda \in \R$ and of radius $r >0$. Fix $\lambda$ and $r$ such that $\I(\lambda,r) \subset$ supp$(\chi)$. Then the function of two variables $\chi(x_1+x_2)$ has support containing the oblique strip $\cup _{t \in \I(\lambda,r)}\{ (s,t-s) : s \in \R \} \subset \R^2$. Let $\psi \in \CC^{\infty}_c(\R)$ be a bump function that equals one on $\I(\lambda,r)$ and zero on $\R \setminus \I(\lambda,2r)$. Let $\theta \in \CC^{\infty}_c(\R)$ be a bump function that equals one on $[-1,1]$ and zero on $\R \setminus [-2,2]$. Let $\Psi_n(x,y) := n^{-1/2} \psi(x+y) \theta(y/n)$. Then $\| \Psi_n\| \equiv \|\psi\| \| \theta \|$. Here $\|\cdot \|$ denotes the norm on $L^2(\R^2)$. Fix $\nu \in \N$, and let $\varphi ^{\nu}_n := (A+\i)^{\nu} \Psi_n$. For $\nu = 0$, clearly $\| \varphi^{\nu}_n\| = \|\Psi_n\|$ is uniformly bounded in $n$ and an easy induction proves it for all fixed values of $\nu \in \N$. Consider now $\phi_n := \chi(H_0) (A+\i)^{-\nu} \varphi^{\nu}_n = \chi(H_0) \Psi_n$. Since $\chi \in \CC_0(\R)$ and $\Psi_n \xrightarrow[]{w} 0$, $\phi_n \xrightarrow[]{w} 0$. 
If $\chi(H_0) (A+\i)^{-\nu} \in \K(L^2(\R^2))$ for some $\nu \in \N$, then the image of the ball $B(0,\sup_{n\geqslant 1} \|\varphi^{\nu}_n\|)$ by this operator is pre-compact in $L^2(\R^2)$, and so there exists $\phi \in L^2(\R^2)$ and a subsequence $(n_k)^{\infty}_{k=1}$ such that $\lim_{k \to +\infty} \|\phi_{n_k} -\phi\| =0$. Since $\phi_{n_k} \xrightarrow[]{w} 0$, it must be that $\phi=0$ since the strong and weak limits coincide and are unique. But this contradicts the fact that $\|\phi_{n_k}\| \geqslant \|\chi \mathbf{1}_{\I(\lambda,r)}(\cdot)\| \| \theta \|$ for all $k \geqslant 1$. So $\chi(H_0) (A+\i)^{-\nu} \not \in \K(L^2(\R^2))$, and this implies that $\chi(H_0) \langle A \rangle ^{-s} \not \in \K(L^2(\R^2))$ for all $s \leqslant \nu$. The result follows by taking adjoints.
\qed
\end{proof}

For what it is worth, we tweak Example \ref{ex:Derivative} to create a system that fits all the assumptions of this article. We state a variation of it and leave the details of the proof to the reader. 

\begin{example}
%\normalfont
Let $\H := L^2(\R^d)$. Let $H_0$ be the operator of multiplication by $x_1 h(x_1)+...+x_d h(x_d)$, where $h \in \CC^{\infty}(\R)$ is a smooth version of the Heaviside function (which is zero below the origin, positive above the origin and strictly increasing). Then $\sigma(H_0) = [0,+\infty)$. In particular, $H_0$ is a positive operator. The conjugate operator is still $A:= \i(\partial / \partial x_1 +...+ \partial / \partial x_d)$. We have $H_0 \in \CC^{\infty}(A)$ and the Mourre estimate holds on all positive bounded intervals. One also verifies that $\{e^{\i tA}\}_{t \in \R}$ stabilizes $\D(H_0)$ (note that $\{e^{\i tA}\}_{t \in \R}$ is the group of translations on $L^2(\R^d)$). Assumptions \ref{item:A1} - \ref{item:A2} are verified. With regard to the compactness issue, Proposition \ref{TwoDimensionKnownResult} holds, but for the two-dimensional result, one must also assume that $\chi$ has non empty support in $(0,+\infty)$. 
\end{example}

\begin{comment}
\begin{example} (The Stark Effect)
\normalfont 
Let $\H := L^2(\R_+^d)$. The Hamiltonian describing a quantum mechanical particle in a constant electric field in the --$x_1$ direction is given by $-\Delta + Ex_1$, where $E>0$ is the strength of the electric field. Let $H_0$ be the closure of this operator on $\mathcal{S}(\R_+^d)$, the Schwartz class. Then $H_0$ is self-adjoint (\cite[Theorem 7.1]{CFKS}). The spectrum of $H_0$ is purely absolutely continuous and bounded below. Let $A := \i \partial / \partial x_1$. We have that $H_0 \in \CC^{\infty}(A)$, and $[H_0,\i A]_{\circ} = I$ implies the Mourre estimate on all bounded intervals contained in the spectrum of $H_0$. Consider the decomposition $L^2(\R_+^d) = L^2(\R_+) \otimes L^2(\R_+^{d-1})$ according to the coordinate decomposition $x=(x_1,x_{\perp})$ in position space and $p= (p_1,p_{\perp})$ in momentum space. Then on $\S(\R_+^d)$ the following formula holds: 
\begin{equation*}
H_0 = e^{-\i p_1 ^3/3}(p_{\perp}^2 +x_1)e^{\i p_1^3/3}.
\end{equation*}
It follows that $\chi(H_0) = e^{-\i p_1 ^3/3}\chi(p_{\perp}^2 +x_1)e^{\i p_1^3/3}$ for all $\chi \in \CC^{\infty}_c(\R)$. We also have $\tau(A) = e^{-\i p_1 ^3/3} \tau(p_1) e^{\i p_1^3/3}$ for all $\tau \in \CC^{\infty}_c(\R)$. Thus $\tau(A) \chi(H_0)$ is unitarily equivalent to $\tau(p_1) \chi(p_{\perp}^2 +x_1)$. This operator is compact in dimension one, as in Proposition \ref{KnownCompactResult}, but isn't in higher dimensions.
\end{example}
\end{comment}

Our next model is the continuous Laplacian. We refer to Section \ref{ex:3} for a description of the model. The situation is the same as with the preceding example: compactness in dimension one, non-compactness in higher dimensions. 

\begin{example} [Continuous Laplacian with generator of dilations]
%\normalfont
\label{ex:contLaplacian}
Let $\H := L^2(\R^d)$, $H_0 := -\Delta$ be the Laplacian and $A := -\i (x\cdot \grad +\grad\cdot x)/2 = -\i (2x\cdot \grad +d)/2$ be the generator of dilations. We will be making use of the Fourier transform on $L^2(\R^d)$ given by 
\begin{equation}
\label{FourierTransform}
(\mathcal{F}\psi)(\xi) = (2\pi)^{-d/2} \int _{\R^d} \psi(x) e^{-\i \xi \cdot x} dx.
\end{equation}
Note that $\F A \F^{-1} = -A =  \sum_{i=1}^d \i (\xi_i \partial / \partial \xi_i +  \partial / \partial \xi_i \xi_i)/2$ and $\F H_0 \F^{-1} = |\xi|^2 := \sum_{i=1}^d \xi_i^2$.

\end{example}

\begin{proposition}
\label{Prop:DeltaA}
Let $H_0$ and $A$ be those from Example \ref{ex:contLaplacian}. We have $\tau(A)\chi(H_0)  \in \K(L^2(\R))$ for all $\tau,\chi \in \CC_0(\R)$, with $\chi$ supported away from zero.
\end{proposition}
\noindent \textit{First proof:}
Let $Q$ be the operator of multiplication by the variable $x$ and $P := -\i d/d x$. Let $\chi \in \CC_c^{\infty}(\R)$ be supported away from zero. Let $(A+\i)^{-1} \chi(H_0) =  (A+\i)^{-1}\chi_1(P)$, where $\chi_1 := \chi \circ \sigma$ and $\sigma(\xi)=\xi^2$. We implement a binary relation $\approx$ on $\B(L^2(\R))$ whereby two operators are equivalent if their difference is a compact operator. We have: 
\begin{align*}
(A+\i)^{-1}  \chi(H_0) &= (A+\i)^{-1}  \chi_1(P)  (Q+\i)(Q+\i)^{-1}  \\
& \approx (A+\i)^{-1}  \chi_1(P) Q(Q+\i)^{-1} \\
&\approx (A+\i)^{-1}  Q \chi_1(P)  (Q+\i)^{-1} \\
&= (A+\i)^{-1} (QP) \chi_2(P) (Q+\i)^{-1} \\
&\approx  (A+\i)^{-1} (A+\i)  \chi_2 (P) (Q+\i)^{-1} \approx 0.
\end{align*}
Note the use of Proposition \ref{TwoDimensionKnownResult} each time a compact operator was removed. In the third step we used that $[\chi_1(P),Q]_{\circ} (Q+\i)^{-1} = \chi_1'(P) (Q+\i)^{-1} \approx 0$. In the fourth step we took advantage of the fact that $\chi_1$ is supported away from zero to let $\chi_2(P) := P^{-1}\chi_1(P)$ and thereby allow to recreate $A:= (QP + PQ)/2 = QP -\i/2$. 

Thus we have shown that $(A+\i)^{-1} \chi(H_0) \in \K(L^2(\R))$. It follows that $(A-z)^{-1} \chi(H_0) \in \K(L^2(\R))$ for all $z \in \C \setminus \R$. Note that the functions $\{(x-z)^{-1} : z \in \C \setminus \R\}$ and $\CC_{c}^{\infty}(\R)$ are dense in $\CC_0(\R)$ with respect to the uniform norm. Since $H_0$ and $A$ are self-adjoint operators, they are unitarily equivalent to a multiplication operator by a real-valued function in some appropriate $L^2(M)$ space. The norm of a multiplication operator from $L^2(M)$ to $L^2(M)$ is equal to the uniform norm of the multiplication function. Two limiting arguments, one for the $H_0$ first and then one for $A_0$, or vice-versa, extends the compactness to $\tau(A)\chi(H_0)$ as in the statement of the Proposition. 
\qed

\noindent \textit{Second proof:}
We see that $\mathcal{F}(A - \i/2)^{-1}\chi(H_0)\mathcal{F}^{-1}$ is an integral transform acting in the momentum space as follows: 
\begin{equation*}
L^2(\R) \ni \varphi \mapsto (\mathcal{F}(A - \i/2)^{-1}\chi(H_0)\mathcal{F}^{-1} \varphi)(\xi) = \frac{\i}{\xi} \int _{0} ^{\xi} \chi(t^2) \varphi(t) dt \in L^2(\R).
\end{equation*}
The fact that $\chi$ is supported away from zero is crucial here. Moreover, if $\chi \in \CC^{\infty}_c(\R)$, then this integral transform is Hilbert-Schmidt and there is $c>0$ such that
\begin{equation*}
\|(A-\i /2)^{-1} \chi(H_0)\|_{HS}^2 = \int_{\R} \int_{\R} \mathbf{1}_{(0,\xi)}(t) \xi^{-2} |\chi(t^2)|^2 dt d\xi \leqslant c \|\chi\|^2_2.
\end{equation*}
In particular, $(A-\i /2)^{-1} \chi(H_0)$ is compact. One extends the compactness to operators of the form $\tau(A)\chi(H_0)$ as in the statement of the Proposition using the same limiting argument explained in the first proof.
\qed

To complete the one-dimensional picture, we mention for what it is worth that it is possible to show that $(A+\i)^{-1}\chi(H_0)  \not \in \K(L^2(\R))$ for any $\chi \in \CC^{\infty}_c(\R)$ with $\chi(0) \neq 0$. We now turn to the multi-dimensional case.

\begin{proposition}
\label{dnotcompact}
Let $H_0$ and $A$ be those from Example \ref{ex:contLaplacian}. If $d \geqslant 2$, then $\langle A \rangle ^{-s} \chi(H_0) \not \in \K(L^2(\R^d))$ for any $\chi \in \CC^{\infty}_c(\R)$ whose support is non-empty in $(0,+\infty)$ and for any $s \in \R$. 
\end{proposition}
\begin{proof}
Let $\I(\lambda, r)$ denote the interval of radius $r>0$ centered at $\lambda$. There are $\lambda \in (0,+\infty)$ and $r>0$ such that $\I(\lambda,r) \subset (0,+ \infty)$ and $m := \inf_{x \in \I(\lambda,r)} |\chi(x)| >0$.  Consider the constant energy curves
\[ \{(\xi_1,...,\xi_d) \in \R^d : E = \xi_1^2+...+\xi_d^2 \}. \]
For $d=2$, these are just circles centered at the origin. Forth we work in dimension two to keep the notation clean, but the necessary adjustments are obvious for $d \geqslant 2$. The support of the function of two variables $\chi(\xi_1^2+\xi_2^2)$ contains the annulus obtained by rotating $\I (\lambda',r')$ about the origin, where
\begin{equation*}
\lambda' := (\sqrt{\lambda+r}+\sqrt{\lambda-r})/2, \quad r' := (\sqrt{\lambda+r}-\sqrt{\lambda-r})/2.
\end{equation*}
Let $\psi_1,\psi_2 \in \CC_c^{\infty}(\R)$ be any bump functions verifying : a) $\psi_1(0) \neq 0$, b) supp($\psi_1$) $= [-1,1]$, c) supp($\psi_2$) $\subset \I (\lambda', r'/2)$, and d) $\|\psi_i \|= 1$, where $\| \cdot \|$ denotes the $L^2$ norm. Now let $\Psi_n(\xi_1,\xi_2) := \sqrt{n} \psi_1(\xi_1 n) \psi_2(\xi_2)$. Then $\| \Psi_n\|=1$ for all $n \geqslant 1$, and $\Psi_n \xrightarrow[]{w} 0$. Also, for $n$ sufficiently large, $\Psi_n$ is supported in the aforementioned annulus. Now fix $\nu \in \N$ and let $\varphi_n^{\nu} := \F(A+\i)^{\nu}\F^{-1} \Psi_n$.  Then for $\nu =0$, $\|\varphi_n ^{\nu}\| = \| \Psi_n \| \equiv 1$, while for $\nu = 1$,
\begin{equation*}
\varphi_n^{\nu}(\xi_1,\xi_2) = -2n^{3/2}\i \xi_1\psi_1'(\xi_1 n)\psi_2(\xi_2) - 2n^{1/2} \i \psi_1(\xi_1 n) \xi_2 \psi_2'(\xi_2) - \i \Psi_n(x),
\end{equation*}
and we see that $\|\varphi_n^{\nu}\|$ is uniformly bounded in $n$. A simple induction on $\nu$ shows that for every fixed value of $\nu \in \N$, $\|\varphi_n ^{\nu}\|$ is uniformly bounded in $n$. Consider $\phi_n := \F \chi(H_0)(A+\i)^{-\nu} \F^{-1} \varphi_n ^{\nu} = \F \chi(H_0) \F^{-1} \Psi_n$.
If $\F \chi(H_0)(A+\i)^{-\nu}\F^{-1} \in \K(L^2(\R^2))$ for some value of $\nu$, the image of the ball $B(0, \sup_{n\geqslant 1} \|\varphi_n ^{\nu}\|)$ by this operator is pre-compact in $L^2(\R^2)$, and so there exists $\phi \in L^2(\R^2)$ and a subsequence $(n_k)_{k=1}^{\infty}$ such that $\lim_{k \to +\infty} \|\phi_{n_k} - \phi\| =0$. Since $\phi_{n_k} \xrightarrow[]{w} 0$, it must be that $\phi = 0$ since the strong and weak limits coincide and are unique. But this contradicts the fact that $\|\phi_{n_k} \| \geqslant  m \|\Psi_{n_k} \| = m > 0$ for all $k\geqslant 1$.  So $\chi(H_0)(A+\i)^{-\nu} \not \in \K(L^2(\R^2))$ and this implies that $\chi(H_0)\langle A \rangle ^{-s} \not \in \K(L^2(\R^2))$ for all $s \leqslant \nu$. The result follows by taking adjoints.
\qed
\end{proof}
A nice corollary of Proposition \ref{dnotcompact} that deserves a mention is the following. It uses Proposition \ref{KnownFourierDecay}. The result can also be proven to hold in dimension one.
\begin{corollary} 
Let $A$ be that from Example \ref{ex:contLaplacian}. Let $d \geqslant 2$. Then for all $(s,\epsilon) \in \R \times (0,+\infty)$, $\langle A \rangle ^{-s} \langle Q \rangle ^{\epsilon} \not \in \B(L^2(\R^d))$. 
\end{corollary}

\begin{example} [Continuous Laplacian with Nakamura's conjugate operator]
%\normalfont
In \cite{N}, Nakamura presents an alternate conjugate operator to the continuous Laplacian $H_0$. Let $\beta > 0$. In momentum space it reads
\[ \F A \F^{-1} := \frac{\i}{2\beta} \sum _{i=1}^d \left( \sin(\beta \xi_i) \frac{\partial}{\partial \xi_i} + \frac{\partial}{\partial \xi_i} \sin(\beta \xi_i)  \right).\]
Under some conditions on the potential $V$, it is shown that the Mourre theory holds for $H := H_0 + V$ with respect to $A$ on the interval $(0, (\pi/\beta)^2/2)$. We refer also to \cite{Ma} for a generalization of this conjugate operator and a more in-depth discussion. An argument as in Propositions \ref{dnotcompact} and \ref{discretenotcompact} shows that, for $d \geqslant 2$, $\langle A \rangle ^{-s} \chi(H_0) \not \in \K(L^2(\R^d))$ for all $\chi \in \CC_0(\R)$ and $s \in \R$.
\end{example}

Our last example is the discrete Laplacian on $\Z^d$. We refer to Section \ref{ex:2} for the details on the model.
\begin{example} [Discrete Schr\"odinger operators]
%\normalfont
\label{ex:discSchro}
Let $\H := \ell^2(\Z^d)$, $H_0 := \Delta$ be the discrete Laplacian and $A$ be its conjugate operator as in Example \ref{ex:2}. Let 
\begin{equation*}
\ell^2(\Z^d) \ni u \mapsto (\F u)(\theta) = (2\pi)^{-d/2} \sum_{n \in \Z^d} u(n) e^{\i \theta \cdot n} \in L^2([-\pi,\pi]^d)
\end{equation*} 
be the Fourier transform. We recall that $H_0$ is unitarily equivalent to the operator of multiplication by $\sum_{i=1}^d (2-2\cos(\theta_i))$ and that $A$ is unitarily equivalent to the self-adjoint realization of the operator $\i  \sum _{i=1} ^d (\sin(\theta_i) \partial / \partial \theta_i + \partial / \partial \theta_i \sin(\theta_i))$, which we denote by $A_{\F}$. 
\end{example}

\begin{proposition} 
Let $H_0$ and $A$ be those from Example \ref{ex:discSchro}. If $d=1$, then $\tau (A) \chi (H_0) \in \K(\ell^2(\Z))$ for all $\tau \in \CC_0(\R)$ and $\chi \in \CC([0,4])$ supported away from zero and four. 
\end{proposition}
\begin{proof}
Using simple techniques from the theory of first order differential equations, we see that $\chi(H_0) (A+\i)^{-1}$ is a Hilbert-Schmidt integral transform acting as follows:
\begin{equation*}
L^2([-\pi,\pi]) \ni \psi \mapsto (\F \chi(H) (A+\i)^{-1} \F^{-1}\psi)(\theta) = \frac{1}{2\i \sin(\theta/2)} \int _0 ^{\theta} \frac{\sin(t/2)}{\sin(t)} \chi \left(2-2\cos(t)\right) \psi(t) dt.
\end{equation*} 
Note that it is crucial that $\chi(2-2\cos(t))$ be supported away from zero and $\pm \pi$. 
\qed
\end{proof}

%%%%%%%%%%%%%%%%%%%%%%%%%%%%%%%%%%%%%%%%%%%%
%%%%%%%%%%%%%%%%%%%%%%

%%%%%%%%  GRAPH
%%%%%%%%%%%%%%%%%%%%%%
%%%%%%%%%%%%%%%%%%%%%%
%%%%%%%%%%%%%%%%%%%%%%%%%%%%%%%%%%%%%%%%%%%%%%%%

\begin{figure}[ht]
\begin{tikzpicture}[domain=-5:5]
\begin{axis}
[grid = major, 
clip = true, 
clip mode=individual, 
axis x line = middle, 
axis y line = middle, 
xlabel={$\theta_1$}, 
xlabel style={at=(current axis.right of origin), anchor=west}, 
ylabel={$\theta_2$}, 
ylabel style={at=(current axis.above origin), anchor=south}, 
domain = -3.14152965:3.14152965, 
xmin = -3.3, 
xmax = 3.3, 
ymin = -3.3, 
ymax = 3.3, 
y=1cm,
x=1cm,
xtick={-3.14159, -1.5708, 1.5708, 3.14159},
xticklabels={$-\pi$, $-\frac{\pi}{2}$, $\frac{\pi}{2}$, $\pi$},
ytick={-3.14159, -1.5708, 1.5708, 3.14159},
yticklabels={$-\pi$, $-\frac{\pi}{2}$, $\frac{\pi}{2}$, $\pi$},
after end axis/.code={\path (axis cs:0,0) node [anchor=north east,yshift=-0.075cm] {0} ;}]
%node [anchor=east,xshift=-0.075cm] {0}
  % energy 1
    \addplot[samples=200, domain = -1.04719755:1.04719755, color = red] {acos(2-cos(deg(x))-0.5)*3.14159265/180};
    \addplot[samples=200, domain = -1.04719755:1.04719755, color = red] {-acos(2-cos(deg(x))-0.5)*3.14159265/180};
  
  % energy 2
    \addplot[samples=200, domain = -1.570796:1.570796, color = red] {acos(2-cos(deg(x))-1)*3.14159265/180};
    \addplot[samples=200, domain = -1.570796:1.570796, color = red] {-acos(2-cos(deg(x))-1)*3.14159265/180};
    
    %energy 3
    \addplot[samples=200, domain = -2.0943951:2.0943951, color = red] {acos(2-cos(deg(x))-1.5)*3.14159265/180};
    \addplot[samples=200, domain = -2.0943951:2.0943951, color = red] {-acos(2-cos(deg(x))-1.5)*3.14159265/180};
    
    %energy 4
    \addplot[samples=200, domain = -3.14159265:3.14159265, color = red] {acos(2-cos(deg(x))-2)*3.14159265/180};
    \addplot[samples=200, domain = -3.14159265:3.14159265, color = red] {-acos(2-cos(deg(x))-2)*3.14159265/180};
    
            % energy 5
    \addplot[samples=200, domain = 1.04719755:3.14159265, color = red] {acos(2-cos(deg(x))-2.5)*3.14159265/180};
    \addplot[samples=200, domain = 1.04719755:3.14159265, color = red] {-acos(2-cos(deg(x))-2.5)*3.14159265/180};
    \addplot[samples=200, domain = -3.14159265:-1.04719755, color = red] {acos(2-cos(deg(x))-2.5)*3.14159265/180};
    \addplot[samples=200, domain = -3.14159265:-1.04719755, color = red] {-acos(2-cos(deg(x))-2.5)*3.14159265/180};
    
        % energy 6
    \addplot[samples=200, domain = 1.570796:3.14159265, color = red] {acos(2-cos(deg(x))-3)*3.14159265/180};
    \addplot[samples=200, domain = 1.570796:3.14159265, color = red] {-acos(2-cos(deg(x))-3)*3.14159265/180};
    \addplot[samples=200, domain = -3.14159265:-1.570796, color = red] {acos(2-cos(deg(x))-3)*3.14159265/180};
    \addplot[samples=200, domain = -3.14159265:-1.570796, color = red] {-acos(2-cos(deg(x))-3)*3.14159265/180};
    
    % energy 7
    \addplot[samples=200, domain = 2.0943951:3.14159265, color = red] {acos(2-cos(deg(x))-3.5)*3.14159265/180};
    \addplot[samples=200, domain = 2.0943951:3.14159265, color = red] {-acos(2-cos(deg(x))-3.5)*3.14159265/180};
     \addplot[samples=200, domain = -3.14159265:-2.0943951, color = red] {acos(2-cos(deg(x))-3.5)*3.14159265/180};
     \addplot[samples=200, domain = -3.14159265:-2.0943951, color = red] {-acos(2-cos(deg(x))-3.5)*3.14159265/180};
    
\end{axis}
\end{tikzpicture}
\caption{Level curves $\{(\theta_1,\theta_2) \in [-\pi,\pi]^2 : E = 2-2\cos(\theta_1) + 2 - 2\cos(\theta_2) \}$ of constant energy for $d=2$}
\label{fig:levelcurve}
\end{figure}
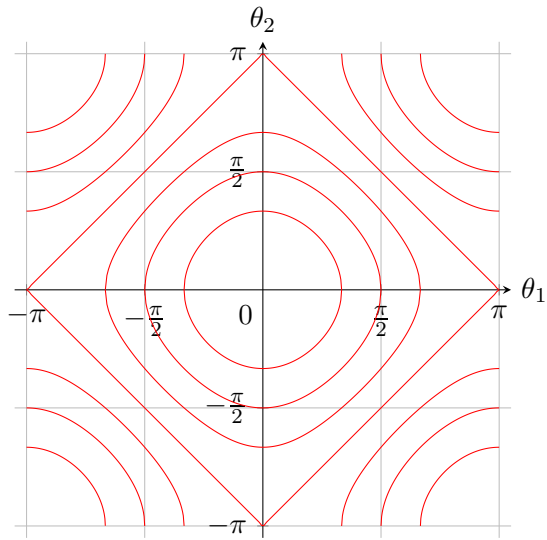

\begin{proposition}
\label{discretenotcompact}
Let $H_0$ and $A$ be those from Example \ref{ex:discSchro}. If $d \geqslant 2$, then $\langle A \rangle ^{-s} \chi (H_0) \not \in \K(\ell^2(\Z^d))$ for all $\chi \in \CC([0,4d])$ with non-empty support in $(0,4d)$, and for all $s \in \R$.
\end{proposition}
\begin{proof}
Let $\lambda \in (0, 4d)$ and $r>0$ be such that $\I(\lambda,r) \subset (0,4d)$ and $m:= \inf_{x \in \I(\lambda,r)} | \chi(x)| > 0$. Fix an energy $E \in \I(\lambda,r)$. Consider the constant energy curves 
\begin{equation*}
\{(\theta_1,...,\theta_d) \in [-\pi,\pi]^d : E = 2-2\cos(\theta_1) + ... + 2-2\cos(\theta_d) \}.
\end{equation*}
For $d=2$, these level curves are drawn in Figure \ref{fig:levelcurve} for various energies in $[0,8]$. Let us proceed in dimension two to keep things simple. The aim is to show that $\F \chi (H_0) \F^{-1} (A_{\F} +\i)^{-\nu}$ is not compact for every fixed value of $\nu \in \N$. Now $\F \chi (H_0) \F^{-1}$ is equal to the operator of multiplication by $\chi(2-2\cos(\theta_1)+2-2\cos(\theta_2))$. The support of this function of two variables contains a neighborhood of a portion of the following vertical axes : $\theta_1=-\pi,0$ or $\pi$. Let $\mathcal{N}$ be such a neighborhood. Let $T$ be one of these three values depending on the situation. We can then create a sequence $\Psi_n(\theta_1,\theta_2) = \sqrt{n} \psi_1((\theta_1-T) n )\psi_2(\theta_2)$ that is supported in $\mathcal{N}$, converges weakly to zero and $\|\Psi_n\| \equiv 1$. Now let $\varphi_n^{\nu} := (A_{\F} + \i)^{\nu} \Psi_n$. Then for every fixed value of $\nu$, $\|\varphi_n ^{\nu}\|$ is uniformly bounded in $n$. The rest of the proof follows the guidelines as that of Proposition \ref{dnotcompact}.
\qed
\end{proof}

Finally, as in the continuous case, we have: 
\begin{corollary} 
Let $A$ be that from Example \ref{ex:discSchro}. Let $d \geqslant 2$. Then for all $(s,\epsilon) \in \R \times (0,+\infty)$, $\langle A \rangle ^{-s} \langle N \rangle ^{\epsilon} \not \in \B(\ell^2(\Z^d))$. 
\end{corollary}

\appendix
\section{Why scattering states evolve where $A$ is prevalent}
\label{Heuristic}
This appendix is based on \cite[Section 3.2]{Go}. We give an idea why it is not unreasonable to expect both purely absolutely continuous spectrum and a propagation estimate under the only assumptions $H \in \CC^1(A)$ and the Mourre estimate \eqref{Mestimate} on $\I$, when $\I$ is void of eigenvalues. Without loss of generality, we may assume that the Mourre estimate for $H$ is strict over the interval $\I$. Given a state $f$ and $f_t := e^{-\i tH}f$ its evolution at time $t \in \R$ under the dynamics generated by the operator $H$, one looks at the Heisenberg picture:
\begin{equation}
\label{Expectation}
\mathcal{A}_f(t) := \langle f_t,A f_t \rangle.
\end{equation}
This is the time-evolution of the expectation value of the observable $A$. Since we are localized in energy in $\I$, and $A$ is generally an unbounded operator, we take $f := \varphi(H) g$, with $g \in \mathcal{D}(A)$ and $\varphi \in C^{\infty}_c(\mathcal{I})$, the smooth functions compactly supported on the interval $\mathcal{I}$. In addition to imply that $[H, \i A]_{\circ} \in  \mathcal{B}(\mathcal{D}(H),\mathcal{D}(H)^*)$, the assumption $H \in \CC^1(A)$ implies that $e^{-\i tH} \varphi(H)$ stabilizes the domain of $A$, ensuring that $\mathcal{A}_f(t)$ is well defined.
Differentiating \eqref{Expectation} gives
\begin{equation}
\label{DerExpectation}
\mathcal{A}'_f(t) := \langle f_t,[H, \i A]_{\circ} f_t \rangle = \langle f_t, E_{\mathcal{I}}(H) [H, \i A]_{\circ} E_{\mathcal{I}}(H) f_t \rangle .
\end{equation}
By using the strict Mourre estimate and the boundedness of $[H, \i A]_{\circ}$ we get
\begin{equation*}
c \|f\|^2  \leqslant \mathcal{A}'_f(t) \leqslant  k \| f\|^2,
\end{equation*}
where $k := \|[H, \i A]_{\circ}\|_{\mathcal{B}(\mathcal{D}(H),\mathcal{D}(H)^*)}$. Integrating this equation yields
\begin{equation*}
c t \|f\|^2 \leqslant \mathcal{A}_f(t) -\mathcal{A}_f(0) \leqslant k t \| f\|^2, \ \text{for all} \ t \geqslant 0.
\end{equation*}
The transport of the particle is therefore ballistic with respect to $A$. This is characteristic of purely absolutely continuous states and propagation estimates are usually obtained in these circumstances. 

\section{A uniform RAGE Theorem}
\label{RAGEappendix}

We would like to make a relevant observation about the RAGE Theorem that appears to be absent from the literature. A small modification of the proof of \cite[Theorem 5.8]{CFKS} leads to:

\begin{theorem} 
[RAGE]
\label{CFKSRAGE}
Let $H$ be a self-adjoint operator on $\H$. Let $\I$ be a bounded interval. 

\noindent $(1)$\ If $W \in \K(\H)$, 
\begin{align} \label{RAGEsup} 
\sup_{\psi \in \H, \|\psi\|=1}\frac{1}{T} \int_0 ^T \| W e^{-\i tH} P_{\rm{c}}(H) \psi \|^2 \ dt \to 0 \quad \text{as} \ T \to \infty. \end{align}

\noindent $(2)$\  If $W \in \B(\H)$ and is $H$-relatively compact, then for all $\psi \in \H$,
\[ \frac{1}{T} \int_0 ^T \| W e^{-\i tH} P_{\rm{c}}(H) \psi \|^2 \ dt \to 0 \quad \text{as} \ T \to \infty. \]
and
\[ \sup_{\psi \in \H, \|\psi\|=1}\frac{1}{T} \int_0 ^T \| W e^{-\i tH} P_{\rm{c}}(H) E_\I(H) \psi \|^2 \ dt \to 0 \quad \text{as} \ T \to \infty. \]
\noindent $(3)$\  If $W \in \K(\H)$, then
\[ \bigg \| \frac{1}{T} \int_0 ^T e^{\i tH} W e^{-i tH} P_{\rm{c}}(H) \ dt \bigg \| \to 0 \quad \text{as} \ T \to \infty. \]
\end{theorem}
 
The improvement consists in the supremum which is absent in the standard version of the Theorem.
In \cite{CFKS}, they prove (3). They state a weaker version of (1), although their proof gives in fact (1). The first part of (2) is proven in \cite{CFKS}. For the second part, apply (1) with $\tilde\psi :=  (H+\i) E_\I(H) \psi$ and conclude by noticing that $(H+\i) E_\I(H)$ is a bounded operator.

\section{Overview of almost analytic extension of smooth functions}
\label{Appendix}

We refer to \cite{D}, \cite{DG}, \cite{GJ1}, \cite{GJ2}, \cite{HS2} and \cite{M} for more details. We review basic results that are spread out in the mentioned literature. Let $\rho \in \R$ and denote by $\mathcal{S}^{\rho}(\R)$ the class of functions $\varphi$ in $\CC^{\infty}(\R)$ such that 
\begin{equation}
|\varphi^{(k)}  (x)| \leqslant C_k \langle x \rangle ^{\rho-k}, \quad \text{for all} \ k \in \N.
\label{decay1}
\end{equation}
For the purpose of this article we only need the class $\mathcal{S}^{0}(\R)$. This class consists of the smooth bounded functions having derivatives with suitable decay. 
\begin{Lemma} \cite{D} and \cite{DG}
\label{DGDavies}
Let $\varphi \in \mathcal{S}^{\rho}(\R)$, $\rho \in \R$. Then for every $N \in \Z^+$ there exists a smooth function $\tilde{\varphi}_N : \C \to \C$, called an almost analytic extension of $\varphi$, satisfying:
\begin{equation}
\tilde{\varphi}_N(x+\i 0) = \varphi(x) \ \forall x \in \R;
\end{equation}
\begin{equation}
\mathrm{supp} \ (\tilde{\varphi}_N) \subset \{x+\i y : |y| \leqslant \langle x \rangle \};
\end{equation}
\begin{equation}
\tilde{\varphi}_N(x+\i y) = 0 \ \forall y \in \R \ \mathrm{whenever} \ \varphi(x) = 0;
\end{equation}
\begin{equation}
\forall \ell \in \N \cap [0,N], \Bigg| \frac{\partial \tilde{\varphi}_N}{\partial \overline{z}}(x+\i y) \Bigg| \leqslant c_{\ell} \langle x \rangle ^{\rho -1-\ell} |y|^{\ell} \ \mathrm{for \ some \ constants} \ c_{\ell} >0.
\label{dei}
\end{equation}
\end{Lemma}

\begin{Lemma} \cite{GJ1}
\label{extend rho}
Let $\rho \geqslant 0$ and $\varphi \in \mathcal{S}^{\rho}(\R)$. Let $\varphi(A)$ with domain $\mathcal{D}(\varphi(A)) \supset \mathcal{D}(\langle A \rangle ^{\rho})$ be the operator whose existence is assured by the spectral theorem. Then for $f \in \mathcal{D}(\langle A \rangle ^{\rho})$, 
\begin{equation}
\varphi(A)f = \lim \limits_{R \to \infty} \frac{\i}{2\pi} \int_{\C} \frac{\partial (\tilde{\varphi\theta_R})_N}{\partial \overline{z}} (z) (z-A)^{-1} f \ dz \wedge d\overline{z},
\label{gross}
\end{equation}
where $\theta_R(x) := \theta(x/R)$ and $\theta \in \CC_c^{\infty}(\R)$ is a bump function such that $\theta(x) =1$ for $x\in [-1/2,1/2]$ and $\theta(x)=0$ for $x\in \R\setminus [-1,1]$. 
\end{Lemma}

\begin{Lemma}
For $\rho \geqslant 0$ and $\varphi \in \mathcal{S}^{\rho}(\R)$, the following limit exists:
\begin{equation}
\varphi^{(k)}(A)f = \lim \limits_{R \to \infty} \frac{\i (k!)}{2\pi} \int_{\C} \frac{\partial (\tilde{\varphi\theta_R})_N}{\partial \overline{z}} (z) (z-A)^{-1-k} f \ dz \wedge d\overline{z}, \quad \text{for all} \  f \in \mathcal{D}(\langle A \rangle ^{\rho}),
\label{gross2}
\end{equation}
where $\theta$ is the same as in Lemma \ref{extend rho}. Moreover, if $0 \leqslant \rho < k$ and $\varphi^{(k)}$ is a bounded function, then $\varphi^{(k)}(A)$ is a bounded operator and 
\begin{equation}
\label{derivative}
\varphi^{(k)}(A) = \frac{\i (k!)}{2\pi} \int _{\C} \frac{\partial \tilde{\varphi}_N}{\partial \overline{z}} (z) (z-A)^{-1-k} \ dz \wedge d\overline{z}
\end{equation}
holds with the integral converging in norm.
\end{Lemma}

\begin{Lemma} 
\label{didid2}
\cite{GJ2}
Let $s \in [0,1]$ and $\mathbb{D} := \{ (x,y) \in \R^2 : 0 < |y| \leqslant \langle x \rangle \}$. Then there exists $c > 0$ independent of $A$ such that for all $z = x+\i y \in \mathbb{D}$ :
\begin{equation}
\|\langle A \rangle ^s (A-z)^{-1}\| \leqslant c \cdot \langle x \rangle ^s \cdot |y|^{-1}.
\end{equation}
\end{Lemma}

\begin{comment}
\begin{Lemma}
\label{lemBox}
Let $\varphi \in \mathcal{S}^{\rho}$, and let $B_1,...,B_n$ be bounded operators. Then for $s \in [0,1]$ satisfying $s < 1 -(1+\rho)/n$, and any $N \geqslant n$, the following integral
\begin{equation}
\label{kdi}
\const \int_{\C} \frac{\partial \tilde{\varphi}_N}{\partial \overline{z}} \prod\limits_{i=1}^{n} \langle A \rangle ^s (z-A)^{-1} B_i \ \dz
\end{equation}
converges in norm to a bounded operator. In particular, for $\rho =0$ and $n \geqslant 3$, \eqref{kdi} converges to a bounded operator for $s \in [0,2/3)$.
\end{Lemma}
\proof
Combine \eqref{didid2} and \eqref{dei} for $\ell=n$.
\qed
\end{comment}

\begin{proposition} 
\label{Prop3}
\cite{GJ1} 
Let $T$ be a bounded self-adjoint operator satisfying $T \in \CC^1(A)$. Then
for any $\varphi \in \mathcal{S}^{\rho}(\R)$ with $\rho < 1$, $T \in \mathcal{C}^1(\varphi(A))$ and
\begin{equation}
\label{use2}
[T,\varphi(A)]_{\circ} = \const \int_{\C} \frac{\partial \tilde{\varphi}_N}{\partial \overline{z}} (z-A)^{-1}[T,A]_{\circ}(z-A)^{-1} \ \dz.
\end{equation}
\end{proposition}

\end{document}